\documentclass[11pt]{amsart}
\usepackage{epsfig, amsmath, amssymb,latexsym,color,esint,amscd}
\usepackage{hyperref}

\newcommand{\N}{\mathbb{N}}
\newcommand{\R}{\mathbb{R}}

\newcommand{\cH}{\mathbb{H}}
\newcommand{\bH}{\mathcal{H}}
\newcommand{\C}{\mathbb{C}}

\newcommand{\M}{\mathcal{M}}
\newcommand{\X}{\mathcal{X}}

\newcommand{\lp}{\left(}
\newcommand{\rp}{\right)}

\newtheorem{Thm}{Theorem}[section]
\newtheorem{Lem}[Thm]{Lemma}

\newtheorem{Cor}[Thm]{Corollary}
\newtheorem{Prop}[Thm]{Proposition}
\newtheorem{Rem}[Thm]{Remark}

\numberwithin{equation}{section}

\DeclareMathOperator{\im}{Im}

\DeclareMathOperator{\Ric}{Ric}
\DeclareMathOperator{\Res}{Res}
\DeclareMathOperator{\dimension}{dim}

\newcommand{\Hei}{\mathbb{H}}

\newcommand\mpar[1]{\marginpar{\tiny \color{red} #1}}

\begin{document}

\title[An extension problem for the CR fractional Laplacian] {An extension problem for the CR fractional Laplacian}


\author[R. L. Frank]{Rupert L. Frank}
\address{R.F.: Mathematics 235-37, Caltech, Pasadena, CA 91125}
\email{rlfrank@caltech.edu}

\author[M.d.M. Gonzalez]{Mar\'{i}a de Mar Gonz\'{a}lez}

\address{M.G.: Departament de Matem\`{a}tica
Aplicada~I, Universitat Polit\`{e}cnica de Catalunya, Av. Diagonal 647,
08028 Barcelona, Spain}

\email{mar.gonzalez@upc.edu}

\author[D. D. Monticelli]{Dario D. Monticelli}

\address{D.M.: Dipartimento Matematica `F. Enriques', Universit\`{a} degli Studi di Milano, Via Saldini 50, 20133, Milano, Italy}

\email{dario.monticelli@gmail.com}

\author[J. Tan]{Jinggang Tan}

\address{J.T.: Departamento de  Matem\'atica,
Universidad T\'{e}cnica Federico Santa Mar\'{i}a,  Avda. Espa\~na 1680,
Valpara\'{\i}so, Chile}

\email{jinggang.tan@usm.cl}

\date{}
\maketitle

\begin{abstract}
We show that the conformally invariant fractional powers of the
sub-Laplacian on the Heisenberg group are given in terms of the scattering
operator for an extension problem to the Siegel upper halfspace.
Remarkably, this extension problem is different from the one studied, among others, by
Caffarelli and Silvestre.
\end{abstract}


\section{Introduction and statement of results}

There has been a lot of recent interest in the study of CR
manifolds, on one hand because of their puzzling geometry and, on
the other hand, because they serve as abstract models of real
submanifolds of complex manifolds. In particular, orientable CR
manifolds of hypersurface type which are strictly pseudoconvex have
been the subject of many flourishing studies, due also to the many
parallels between their geometry and conformal geometry of
Riemannian manifolds. In this context, the Heisenberg group plays
the same r\^{o}le as $\R^n$ in conformal geometry. Indeed, Folland
and Stein \cite{Folland-Stein:Heisenberg} constructed normal
coordinates which show how the Heisenberg group closely approximates
the pseudohermitian structure of a general orientable strictly
pseudoconvex CR manifold.

The Heisenberg group $\mathbb{H}^n$ arises also in the description
of $n$--dimensional quantum mechanical systems. Moreover there is a
rich and fruitful interplay between sub-Riemannian geometry on Carnot
groups (of which $\mathbb{H}^n$ is one of the most interesting
examples) and control theory in engineering, and there are many works devoted to understanding harmonic analysis on Lie groups.

In this paper we take a closer look at the construction of CR
covariant operators of fractional order on $\mathbb{H}^n$ and on
orientable CR manifolds of hypersurface type which are strictly
pseudoconvex, and how they may be constructed as the
Dirichlet-to-Neumann operator of a degenerate elliptic equation in
the spirit of Caffarelli and Silvestre \cite{CaSi}.

Fractional CR covariant operators of order $2\gamma$,
$\gamma\in\mathbb R$, may be defined from scattering theory on a
K\"ahler-Einstein manifold $\mathcal X$
\cite{Epstein91,HislopPeterTang,Guillarmou08,Gover-Graham:CR-powers},
they are pseudodifferential operators whose principal symbol agrees
with the pure fractional powers of the CR sub--Laplacian
$(-\Delta_b)^\gamma$ on the boundary $\mathcal M=\partial \mathcal
X$. In the particular case of the Heisenberg group $\mathbb H^n$,
they are simply the intertwining operators on the CR sphere
calculated in \cite{BrFoMo,Johnson-Wallach,Branson-Olafsson-Orsted}
using classical representation theory tools. It is precisely the
article by Branson, Fontana and Morpurgo \cite{BrFoMo} that
underlined the importance and nice PDE properties of these
operators.

There is a rich theory of pseudodifferential operators on the
Heisenberg group (see, for instance, \cite{Taylor:book}). In
particular, the fractional sub-Laplacian is the infinitesimal generators of a Levy process
\cite{Applebaum-Cohen}, although with some particularities because
of the extra direction.

Using functional analysis tools, one may formulate an extension
problem to construct the pure fractional powers of the
sub--Laplacian on Carnot groups (see the related work of
\cite{FraFe} and \cite{Lopez-Sire}). An interesting feature of our
point of view is that one needs to use the underlying complex
hyperbolic geometry instead of the abstract functional analysis
theory for the construction of the CR covariant version of these
operators.

\vspace{0,2cm}

We try to make this paper self-contained in order to make it accessible to analysts and do
not assume any prerequisites on CR or complex geometry. In this regard, Sections \ref{section-preliminaries} and \ref{section-Heisenberg} are a summary of standard concepts that are included here for convenience of the reader.  We emphasize, however, that the reader does not need
these concepts for our main results, Theorems 1.1 and 1.2, which only
concern the case of the Heisenberg group.

The $n$-dimensional Heisenberg group $\mathbb{H}^n$
is the set $\R^n\times\R^n\times\R$ endowed with the group law
$$\hat\xi\circ\xi=\big(\hat x+x,\hat y+y,\hat t+t+2(\langle x,\hat
y\rangle_{\R^n}-\langle \hat x, y\rangle_{\R^n})\big),$$ where
$\xi=(x,y,t)$, $\hat\xi=(\hat x,\hat y,\hat t)$ and
$\langle\cdot,\cdot\rangle_{\R^n}$ is the standard inner product in
$\R^n$. $\mathbb{H}^n$ can be regarded as a smooth sub--Riemannian
manifold; an orthonormal frame is given by the Lie left--invariant
vector fields
\begin{equation}\label{2}
X_j=\frac{\partial}{\partial x_j}+2y_j\frac{\partial}{\partial
t},\quad Y_j=\frac{\partial}{\partial
y_j}-2x_j\frac{\partial}{\partial t}\quad\text{for
}j=1,\ldots,n,\quad T_0=\frac{\partial}{\partial t}.
\end{equation}
Given a smooth function on $\cH^{n}$, we define the sub--Laplacian
of $u$ as:
\begin{equation}\label{Laplacian0}\Delta_{b}u=\frac{1}{2}\sum_{j=1}^n\left[X^2_j+Y^2_j\right]u.
\end{equation}
Note that we have replaced the usual factor $1/4$ in front by a $1/2$, which will be more convenient for our purposes.

Introducing complex coordinates $z=x+iy\in\mathbb{C}^n$, the
Heisenberg group may be identified with the boundary
of the Siegel domain $\Omega_{n+1}\subset \mathbb C^{n+1}$, which is
defined as
$$\Omega_{n+1}:=\left\{\zeta=(z_1,\ldots,z_n,z_{n+1})=(z,z_{n+1})\in\C^{n}\times\C
\,|\,q(\zeta)>0\right\}$$ with
\begin{equation}\label{1}
q(\zeta)=\im z_{n+1}-\sum_{j=1}^n|z_j|^2,
\end{equation}
through the map
$(z,t)\in\mathbb{H}^n\mapsto(z,t+i|z|^2)\in\partial\Omega_{n+1}$.
Introducing the coordinates $(z,t,q)\in\C^n\times\R\times(0,\infty)$
on $\Omega_{n+1}$, the Siegel domain is a K\"ahler-Einstein manifold
with K\"ahler form
\begin{equation}\label{Kahler2}
\omega_{+}=-\frac{i}{2}\partial\bar\partial \log
q=\frac{i}{2}\left(\frac{\tfrac{1}{4}dq^2+\theta^{2}}{q^2}+\frac{\delta_{\alpha
\bar{\beta}}\theta^\alpha\wedge\theta^{\bar \beta}}{q}\right),
\end{equation}
where $\theta^\alpha=dz_\alpha$,
$\theta^{\bar{\beta}}=d\bar{z}_\beta$ for $\alpha,\beta=1,\ldots,n$
and where $\theta$ is given by
\begin{equation}\label{theta-introduction}
\theta=\tfrac{1}{2}\big[dt+2\sum_{\alpha=1}^n\left(x_\alpha
dy_\alpha-y_\alpha
dx_\alpha\right)\big]=\tfrac{1}{2}[dt+i\sum_{\alpha=1}^n\left(z_\alpha
d\bar{z}_\alpha-\bar{z}_\alpha dz_\alpha\right)].
\end{equation}
If, instead, one writes the K\"ahler metric for the defining
function $\rho$ with $q=\rho^2/2$, we have
\begin{equation}\label{metric0}
g^{+}=\frac{1}{2}\left(\frac{d\rho^{2}}{\rho^{2}}+\frac{2\delta_{\alpha\bar{\beta}}\theta^{\alpha}\otimes
\theta^{\bar{\beta}}}{\rho^{2}}+ \frac{4\theta^{2}}{\rho^{4}}
\right).\end{equation} In particular, $\Omega_{n+1}$ may be
identified with the complex hyperbolic space $\mathcal H_{\mathbb
C}^m$. Here, and for the rest of the paper, we set
$$m+1.$$

The boundary manifold $\mathcal M=\partial\Omega_{n+1}=\{q=0\}$
inherits a natural CR structure from the complex structure of the
ambient manifold given by the 1-form \eqref{theta-introduction},
that is precisely the contact form that characterizes the Heisenberg
group $\mathbb H^n$ as a sub--Riemannian (CR) manifold, in which the
CR structure is given by the bundle
$\mathcal{H}(\mathbb{H}^n)=\text{span}_{\C}\langle
Z_1,\ldots,Z_n\rangle$, with $Z_j=X_j+iY_j$ and $X_j,Y_j$ defined in
\eqref{2} for $j=1,\ldots,n$. Consequently the Levi distribution on
$\mathbb{H}^n$ is given by $H(\mathbb{H}^n)=\text{span}_\R\langle
X_1,\ldots X_n,Y_1,\ldots,Y_n\rangle$ and the associated
characteristic direction is $$T=2T_0=2\frac{\partial}{\partial t}.$$
In particular, the associated Laplace-Beltrami operator with this
choice of $\theta$ is \eqref{Laplacian0}; this is the reason for our
normalization constant $1/2$.

More generally, one could consider $\mathcal X^{n+1}$ a complex
manifold with strictly pseudoconvex boundary $\mathcal M$ carrying
an approximate K\"ahler-Einstein metric. Then $\mathcal M$ inherits
a CR structure as explained in Section \ref{section-preliminaries}.
An even more general setting would be to take an asymptotically complex hyperbolic (ACH) manifold
$\mathcal X^{m+1}$ with boundary $\mathcal M$. An ACH manifold is
endowed with a metric $g^+$ that  behaves asymptotical like \eqref{metric0} near $\mathcal M$.

In the first setting, scattering theory tells us that, fixed a defining function $q$, for $s\in\mathbb
C$, $\Re(s)\geq \frac{m}{2}$, and except for a set of exceptional
values, given $f$ smooth on $\mathcal M$, the eigenvalue equation
\begin{equation*}
-\Delta_{g^+} u -s(m-s)u=0\qquad\text{in }\mathcal X,
\end{equation*}
has a solution $u$ with the expansion
\begin{equation*}\left\{\begin{split}
&u=q^{(m-s)}F+q^{s} G, \quad \mbox{for some}\quad F,G\in\mathcal C^\infty(\bar\X),\\
& F|_{\mathcal M}=f.
\end{split}\right.\end{equation*}
The scattering operator is defined as
$$S(s): \mathcal C^\infty(\mathcal M) \to \mathcal C^\infty(\mathcal M)$$
by
$$S(s)f:=G|_{\mathcal M}.$$
It can be shown that $S(s)$ is a meromorphic family of pseudodifferential
operators in the $\Theta$-calculus of \cite{Epstein91} on $\mathcal
M$ of order $2(2s-m)$, and it is self-adjoint when $s$ is real.

For $\gamma\in(0,m)\backslash\mathbb N$, we set
$s=\frac{m+\gamma}{2}$. We define the CR fractional sub--Laplacian
on $(\mathcal M,[\theta])$ by
\begin{equation}\label{operators}P_\gamma^\theta f=c_\gamma S(s)f,\end{equation}
for a constant
\begin{equation}\label{constant}
c_\gamma=2^\gamma\frac{\Gamma(\gamma)}{\Gamma(-\gamma)}.
\end{equation}
Note that $c_\gamma<0$ for $0<\gamma<1$. It is proven then that $P_\gamma^\theta$ is a pseudodifferential operator of order
$2\gamma$, whose principal symbol is given by \eqref{symbol} below.
The main property of the operator is its CR covariance.
Indeed, given another conformal representative $\hat
\theta=w^{\frac{2}{m-\gamma}}\theta$ which identifies the CR
structure of $\mathcal{M}$, the corresponding operator is given
by
\begin{equation}\label{CRcovariance}
P_\gamma^{\hat\theta}(\cdot)=w^{-\frac{m+\gamma}{m-\gamma}}P_\gamma^\theta(w\,\cdot).\end{equation}
Then one may define the fractional $Q$-curvature as
$$Q_\gamma^\theta=P_\gamma^\theta(1),$$
which has interesting covariant properties.

For integer powers $\gamma\in\mathbb N$, one may still define the
operators by taking residues of the scattering operator. In
particular, for $\gamma=1$, one obtains the CR Yamabe operator of
Jerison-Lee \cite{Jerison-Lee:Yamabe}
$$P_1^\theta=-\Delta_b+\frac{n}{2(n+1)}R_\theta,$$
where $\Delta_b$ is the sub--Laplacian on $(\M,\theta)$ and
$R_\theta$ the Webster curvature. In the Heisenberg group case, with
the 1-form $\theta$ introduced in \eqref{theta-introduction}, we can
write explicitly
$$P_1^\theta=-\Delta_b, \quad P_2^\theta= (\Delta_b)^2+T^2,\quad
P_k^\theta=\prod_{l=1}^k (-\Delta_b+i(k+1-2l)T).$$ These are
precisely the GJMS \cite{Graham-Jenne-Mason-Sparling} operators in
the CR case from \cite{Gover-Graham:CR-powers}.

Our first theorem is the precise statement of the extension problem
on the Heisenberg group $(\mathbb H^n,[\theta])$ for $\theta$ given
by \eqref{theta-introduction}. This relation allows to treat the CR
fractional sub--Laplacian as a boundary operator and, in particular,
one recovers the formula for CR covariant operators on the
Heisenberg group from \cite{BrFoMo}.

\begin{Thm}\label{thm1}
Let $\gamma\in (0,1)$, $a=1-2\gamma$.
For each $f\in \mathcal C_0^\infty(\mathbb H^n)$, there exists a unique solution $U:=\mathcal E_\gamma f$ for the extension problem
\begin{equation}\label{equation-CafSil}
\left\{\begin{split}
\partial_{\rho\rho} U+a\rho^{-1}\partial_\rho U +\rho^2\partial_{tt}U+\Delta_b U&=0,\quad \mbox{in }\Omega_{n+1}\simeq\mathbb H^n\times\mathbb R^+,\\
U&=f, \quad\mbox{on
}\partial\Omega_{n+1}=\{\rho=0\}\simeq\mathbb{H}^n.
\end{split}\right.
\end{equation}
Moreover,
\begin{equation}\label{7}
P_\gamma^\theta f=\frac{c_\gamma}{\gamma2^{1-\gamma}}\lim_{\rho\to
0} \rho^a
\partial_\rho U.
\end{equation}
As a consequence, one recovers the symbol of the operator
\begin{equation}\label{symbol}
 P_\gamma^\theta = (2 |T|)^\gamma \ \frac{\Gamma\left( \frac{1+\gamma}{2} + \frac{-\Delta_b}{2|T|}\right)}{\Gamma\left( \frac{1-\gamma}{2} + \frac{-\Delta_b}{2|T|}\right)} \,.\end{equation}
\end{Thm}

Uniqueness is understood in the natural corresponding Sobolev space, see
\eqref{6} below.

The extension problem \eqref{equation-CafSil} is
similar to the extension considered by  Caffarelli and Silvestre \cite{CaSi} for the construction of the fractional Laplacian in the Euclidean case, but surprisingly with the additional term in the $t$-direction  $\rho^2\partial_{tt} U$
that appears when one considers the CR sub--Laplacian.

Although here we have concentrated on the Heisenberg group as the
boundary of the Siegel domain, on more general ACH manifolds we have
the same type of results, although lower order terms appear in the
extension problem \eqref{equation-CafSil}. This is exactly what
happens in the real case (see \cite{Chang-Gon}), and precisely those
lower order terms carry a lot of  geometric information.

The main idea in the proof of Theorem \ref{thm1} is to use the group
Fourier transform on the Heisenberg group in order to follow the
arguments used in the real case by \cite{CaSi}. Their idea of
reducing the equation to an ODE in  $\rho$ still holds, although we
need to take care of the extra direction $\partial_t$. Moreover, the
Plancherel identity for the Heisenberg Fourier transform allows to
prove an energy identity that yields a sharp trace Sobolev
embedding.

Before we state our second theorem precisely, we need to introduce
the following notions. Set
\begin{equation}\label{6}
\mathcal A_\gamma[U] = \int_{\Omega_{n+1}} \rho^a\bigg[
|\partial_\rho U|^2 + \rho^{2} |\partial_t U|^2 + \tfrac{1}{2}
\sum_{j=1}^n \left( |X_j U|^2 + |Y_j U|^2 \right) \bigg] \, dx \,dy
\,dt \,d\rho
\end{equation}
for functions $U$ on $\overline{\Omega}_{n+1}\simeq \mathbb
H^n\times[0,\infty)$; this is the weighted Dirichlet energy in the
extension. We define the space $\dot{H}^{1,\gamma}(\Omega_{n+1})$ as
the completion of $\mathcal C_0^\infty(\overline{\Omega}_{n+1})$
with respect to $\mathcal A_\gamma^{1/2}$.

Similarly, we denote by $\dot S^\gamma(\Hei^n)$ the closure of
$\mathcal C_0^\infty(\Hei^n)$ with respect to the quadratic form
$a_\gamma$ associated to the symbol \eqref{symbol} through Fourier
transform, which is defined by
\begin{equation}\label{5}
a_\gamma[f]=\frac{1}{2^\gamma}\int_{\Hei^n}f P^\theta_\gamma f\,
dxdydt
\end{equation}
see also Section \ref{section-trace}. These are the fractional
analogues of the Sobolev spaces on $\mathbb H^n$ introduced by
Folland and Stein \cite{Folland-Stein:Heisenberg}.

\begin{Thm}\label{thm2}
Let $\gamma\in(0,1)$. Then there exists a unique linear bounded
operator $\mathcal{T}:\dot{H}^{1,\gamma}(\Omega_{n+1}) \to \dot
S^\gamma(\Hei^n)$ such that $\mathcal{T}U(\cdot)=U(\cdot,0)$ for all
$U\in \mathcal C_0^\infty(\overline{\Omega_{n+1}})$. Moreover, for
any $U\in \dot{H}^{1,\gamma}(\Omega_{n+1})$ one has
$$
\mathcal A_\gamma[U] \geq2^{1-\gamma} \gamma
\frac{\Gamma(1-\gamma)}{\Gamma(1+\gamma)}\ a_\gamma[\mathcal{T}U]
\,.
$$
Equality is attained if and only if $U=\mathcal E_\gamma f$ for some $f\in \dot S^\gamma(\Hei^n)$.
\end{Thm}

One may use the Cayley transform to translate Theorem \ref{thm2} to
the CR sphere $\mathbb S^{2m-1}$ as the boundary of the complex
Poincar\'e ball $\mathcal H^{m}_{\mathbb C}$. Indeed, the Cayley
transform $\Psi_c$ is a biholomorphism between the unit ball in
$\C^{n + 1}$ and the Siegel domain; when restricted to the
respective boundaries it gives a CR equivalence between $\mathbb
S^{2m-1}$ minus a point and $\mathbb{H}^n\simeq\partial\Omega_{n+1}$
(which is the CR equivalent of the stereographic projection of the
unit sphere on the Euclidean space), see also section
\ref{section-Heisenberg} and \cite{Jerison-Lee:Yamabe}. Combining
Theorem \ref{thm2} with the fractional Sobolev embedding on the
Heisenberg group from \cite{Frank-Lieb} we obtain the following energy
inequality:
\begin{Cor}\label{cor-embedding}There exists a sharp constant $S(n,\gamma)$ such that for every
$U\in \dot H^{1,\gamma}(\mathcal H^m_{\mathbb C})$,
$f=\big(\mathcal{T}(U\circ\Psi_c^{-1})\big)\circ\Psi_c$ defined on
$\mathbb S^{2m-1}$,
\begin{equation}\label{Sobolev}
\|f\|_{L^{2^*}(\mathbb S^{2n+1})}^2\leq S(n,\gamma)\mathcal
A_{\gamma}[U\circ\Psi_c^{-1}]
\end{equation}
for $2^*=\frac{2m}{m-\gamma}$. We have equality in \eqref{Sobolev}
if and only if and $U\circ\Psi_c^{-1}=\mathcal E_\gamma
(f\circ\Psi_c^{-1})$ and $f\circ\Psi_c^{-1}$ comes from a CR
transformation of the CR sphere \eqref{Jacobian}, i.e.
$f\circ\Psi_c^{-1}:\mathbb{H}^n\rightarrow\mathbb{R}$ with
$$f\big(\Psi_c^{-1}(z,t)\big)=\lp\frac{1}{(1+|z|^2)^2+t^2}\rp^\frac{m-\gamma}{2}$$ up to left translations, dilations and multiplication by a
constant.
\end{Cor}

The importance of this Corollary is that it constitutes the first
step in the resolution of the fractional CR Yamabe problem (see
\cite{Gonzalez-Qing} and Section \ref{section-Yamabe} for a short
discussion).

In the last part of this paper (Section \ref{Section-further}), we
explore the quaternionic setting and show that both Theorems
\ref{thm1} and \ref{thm2} can be proven in a similar manner  when
one considers the quaternionic Heisenberg group $\mathcal Q^n$ as
the boundary of the quaternionic hyperbolic space $\mathcal
H^m_{\mathbb Q}$. In particular, the extension problem
\eqref{equation-CafSil} is the same one provided one replaces
$\partial_{tt}$ by the Laplacian in the extra three $t$-directions. This
happens because of the underlying Lie group structure and, in
general, this general construction is possible in symmetric spaces
of rank one, which are the real, complex, quaternionic hyperbolic
spaces $\mathcal H^m_{\mathbb R}$, $\mathcal H^m_{\mathbb C}$,
$\mathcal H^m_{\mathbb Q}$, respectively, and the octonionic
hyperbolic plane $\mathcal H^2_{\mathbb O}$.

There are still many open questions in this field. From the PDE
point of view, an extremely interesting problem is to establish a
good elliptic theory for \eqref{equation-CafSil}, which is fourth
order because of the term $\partial_{tt}$. Indeed the vector field
$T_0=\frac{\partial}{\partial t}$ is a second order differential
operator in the CR structure of $\mathbb{H}^n$, as it is obtained
from first order commutators of the vector fields
$\{X_j,Y_j\}_{j=1,\ldots,n}$. See also sections
\ref{section-preliminaries} and \ref{section-Heisenberg}.

From the complex geometry point of view, observe that two defining
functions $\varphi$ and $\psi$ generate the same K\"ahler metric in
$\mathcal X$ if and only if $\varphi = e^F\psi$ for a pluriharmonic
function $F$, i.e. �$\partial\bar\partial F = 0$. If
$\theta_\varphi$ and $\theta_\psi$ are the corresponding
pseudohermitian structures on $\mathcal M=\partial\mathcal{X}$ then
$\theta_\varphi = e^f \theta_\psi$, where $f = F|_{\mathcal M}$. In
this case, Branson's $Q$-curvature (that corresponds to our case
$\gamma=m$, see \eqref{Branson-curv}), vanishes identically and thus
it is not an interesting object when restricting to pluriharmonics.
Recently, \cite{Hirachi:Q-prime} (see also \cite{Case-Yang}) has
constructed new GJMS operators $P'_k$ using the original
\cite{Graham-Jenne-Mason-Sparling} ambient metric construction. It
would be interesting to see those from the scattering and extension
point of view.

Fractional order CR operators on the Heisenberg group were
introduced in R. Graham's thesis
\cite{Graham:thesisI,Graham:thesisII} as obstructions to the
resolution of an eigenvalue problem. His construction works also for
more general operators
$$\mathcal L_\alpha:=\frac{1}{2}\sum_{j=1}^n [X_j^2+Y_j^2]+i\alpha T,\quad \text{for }\alpha\in\mathbb
R,$$ which up to multiplication by a constant are the only linear
differential operators on $\mathbb{H}^n$ which are second order with
respect to the natural dilation structure of the Heisenberg group,
that are invariant with respect to left translations and that are
invariant under the action of unitary transformations of
$z=x+iy\in\C^n$. One could try to understand what type of new
problems appear.

\section{Preliminaries}\label{section-preliminaries}

Most of this section is taken from the well written introduction in
\cite{HislopPeterTang}, but we recall it here for convenience of the
reader. Two good survey references on Heisenberg groups and CR
manifolds are \cite{Beals-Fefferman-Grossman},
\cite{Dragomir-Tomassini}. We also point out the book on complex
hyperbolic geometry \cite{Goldman:book}.

\subsection{CR geometry}

 A CR-manifold is a smooth manifold $\M$ equipped with a distinguished
subbundle $\mathcal{H}$ of the complexified tangent bundle $\C
 T\M=T\M\otimes\C$ such that $[\mathcal{H},\mathcal{H}]\subset\mathcal{H}$, i.e.
$\mathcal{H}$ is closed with respect to the Lie bracket and hence is
formally integrable, and $\mathcal{H}\cap\bar{\mathcal{H}}=\{0\}$,
i.e. $\mathcal{H}$ is almost Lagrangian. The subbundle $\mathcal{H}$
is called a CR structure on the manifold $\M$. An abstract CR manifold will
be called of hypersurface type if $\text{dim}_\R \M=2n+1$ and
$\text{dim}_\C\mathcal{H}$.

The maximal complex (or Levi) distribution on the CR manifold $\M$ is
the real subbundle $H(\M)\subset T\M$ given by
$H(\M)=\text{Re}\{\mathcal{H}\oplus\bar{\mathcal{H}}\}$; it carries a
natural complex structure $J_b:H(\M)\rightarrow H(\M)$ defined by
$$J_b(V+\bar{V})=i(V-\bar{V})\qquad\text{for any }V\in\mathcal{H}.$$

If the CR manifold $\M$ is of hypersurface type and oriented, as we
will always assume throughout the paper, it is possible to associate
to its CR structure $\mathcal{H}$ a one--form $\theta$ which is
globally defined on $\M$ such that
\begin{align}\label{dm1}
\text{Ker}(\theta) =H(\M).
\end{align}
This form is unique up to multiplication by a non-vanishing function
in~$\mathcal C^{\infty}(\M)$. More precisely, any two globally
defined one--forms $\theta,\,\hat\theta$ on $\M$ such that
$\text{Ker}(\hat\theta)=\text{Ker}(\theta) =H(\M)$ are related by
\begin{align}\label{dm2}
\theta=f\hat\theta,
\end{align}
for some nowhere zero smooth function $f$ on $\M$. Any one--form
$\theta$ satisfying \eqref{dm1} is called a pseudohermitian
structure on $\M$.

The Levi form $L_\theta$ associated to a pseudohermitian structure
$\theta$ on $\M$ is given by
\begin{equation}\label{dm2a}L_\theta(V_{1},\bar{V}_{2})=-i(d\theta)(V_{1},\bar{V}_{2}),\qquad\forall\,
V_{1},\,V_{2}\in\mathcal{H},\end{equation}
where the form $d\theta$ is extended to
complex vector fields by $\C$--linearity.
The Levi form changes under a change of the pseudohermitian
structure given by \eqref{dm2} as follows
$$L_{\hat\theta}=fL_\theta.$$

We say that an orientable CR manifold of hypersurface type endowed
with a pseudohermitian structure is strictly pseudoconvex if its
corresponding Levi form is strictly positive definite, while we say
that it is nondegenerate if the corresponding Levi form is
nondegenerate (i.e. if $V_1\in\mathcal{H}$ satisfies
$L_\theta(V_{1},\bar{V}_{2})=0$ for all $V_{2}\in\mathcal{H}$, then
$V_{1}=0$).

Given a nondegenerate pseudohermitian structure $\theta$ on an
oriented CR manifold $\M$ of hypersurface type,
$\psi=\theta\wedge(d\theta)^n$ is a volume form on $\M$. Moreover
there is a unique globally defined nowhere zero vector field $T$
tangent to $\M$ such that
\begin{equation}\label{dm3}
\theta(T)=1,\qquad T \rfloor d\theta=0.
\end{equation}
$T$ is called characteristic direction of $(\M,\theta)$ and one can
easily show that
$$T\M=H(\M)\oplus\R T.$$
Moreover, the relations
\begin{equation}\label{Webstermetric-dm}
g_\theta(V_{1},V_{2})=d\theta(V_{1},J_bV_{2}),\qquad g_\theta(V_{1},T)=0,\qquad
g_\theta(T,T)=1
\end{equation}
for every $V_{1},V_{2}\in H(\M)$, where $T$ is the characteristic vector field
associated to $\theta$, define the Webster metric on $\M$. If the
Levi form $L_\theta$ is strictly positive definite, then $g_\theta$
is a Riemannian metric on $\M$, but $g_\theta$ is not a CR
invariant. Next, for any smooth function $u$ on $\M$ we define
$\nabla_{g_\theta} u$ through the relation
\begin{equation}\label{dm7}
g_\theta(\nabla_{g_\theta} u,V)=du(V)
\end{equation}
for any vector field $V$ on $\M$. If $\pi_b$ is the canonical
projection of the real tangent bundle $T\M$ on the Levi
distribution $H(\M)$, the horizontal gradient of a smooth function
$u$ on $\M$ is given by
\begin{equation}\label{dm8}
\nabla_{b}u=\pi_{b}\nabla_{g_\theta} u.
\end{equation}
Moreover the sub--Laplacian operator is defined by setting
\begin{equation}\label{dm5}
\Delta_b u=\text{div}(\nabla_{b}u)
\end{equation}
for any smooth function $u$ on $\M$, where $\text{div}(V)$ is
defined by
\begin{equation}\label{dm6}
\mathcal{L}_V\psi=\text{div}(V)\psi,
\end{equation}
or equivalently by
$$d(V\rfloor \psi)=\text{div}(V)\psi,$$ for any vector field $V$ on
$\M$, where $\psi=\theta\wedge(d\theta)^n$ and $\mathcal{L}$ denotes the
Lie derivative.

In coordinate notation, let $\{W_{\alpha}\}_{\alpha=1}^{n}$ be a
local frame for $\bH$ and let
$W_{\bar{\alpha}}=\overline{W}_\alpha$. Then the vector fields
$\{W_{\alpha},W_{\bar{\alpha}},T\}$ form a local frame for
${\C}T\M$, whose dual coframe
$\{\theta^{\alpha},\theta^{\bar{\alpha}},\theta\}$ is admissible if
$\theta^{\alpha}(T)=0$ for all $\alpha$. The integrability condition
is equivalent to the condition that $d\theta=d\theta^\alpha$ mod
$\{\theta,\theta^\alpha\}$. Thus  the Levi form is written as
$$L_{\theta}=h_{\alpha\bar{\beta}}\theta^{\alpha}\wedge\theta^{\bar{\beta}},$$
for a  Hermitian matrix--valued function $h_{\alpha\bar{\beta}}$.
Its inverse will be denoted by $h^{\alpha\bar\beta}$. We will use
$h_{\alpha\bar\beta}$ and $h^{\alpha\bar\beta}$ to lower and raise
indexes in the usual way. The horizontal gradient and the
sub--Laplacian on $\mathcal M$ may be calculated as \mpar{check}
\begin{equation}\label{sublaplacian}
\nabla_b u=u^{\bar{\alpha}} W_\alpha+u^\beta W_{\bar{\beta}},\qquad
\Delta_b u= {u_\alpha}^{\alpha}+{u_{\bar\beta}}^{\bar\beta},
\end{equation}
respectively, where $u^{\bar{\alpha}}=h^{\gamma
\bar{\alpha}}W_{\bar{\gamma}}u,\,u^\beta=h^{\beta\bar{\gamma}}W_{\gamma}u$.

We recall that any real hypersurface $\mathcal{M}$ in a complex
manifold $\mathcal{X}$ of complex dimension $m+1$ can be seen as a
CR manifold of hypersurface type of dimension $2n+1$, with the CR
structure naturally induced by the complex structure of the ambient
manifold
$$\mathcal{H} = T ^{1,0}(\mathcal{X})\cap
\mathbb{C}T\mathcal{M},$$ where $T ^{1,0}(\mathcal{X})$ denotes the
bundle of holomorphic vector fields on $\mathcal{X}$.

\subsection{Complex manifolds with CR boundary}

Now suppose that ${\mathcal{X}}$ is a compact complex manifold of
dimension $m+1$ with boundary $\partial {\mathcal{X}}=\M$. As we have just mentioned, the boundary manifold inherits a natural CR-structure from the
ambient manifold. We recall the following facts from
\cite{GrahamLee,Lee-Melrose}, where they give a precise description
of the asymptotic behavior near the boundary.

Let  $\varphi$ be a negative smooth function on ${\mathcal{X}}$. We
say that $\varphi$ is a defining function for $\M$ if  $\varphi<0$
in the interior of $\mathcal X$, $\varphi=0$ on $\M$,
$|d\varphi(p)|\neq 0$ for all $p\in \M$. We further suppose that
$\varphi$ has no critical points in a collar neighborhood
$\mathcal{U}$ of $\M$ so that the level sets
$\M^{\varepsilon}=\varphi^{-1}(-\varepsilon)$ are smooth manifolds
for all $\varepsilon$ sufficiently small.

Associated to the defining function $\varphi$ we define the K\"ahler form
 \begin{equation}\label{Kahler-form0}
  \omega_{+}=-\frac{i}{2}\partial\bar{\partial}\log (-\varphi)=\frac{i}{2}\left(\frac{\partial\bar{\partial}\varphi}{-\varphi}+
  \frac{\partial\varphi\wedge\bar{\partial}\varphi}{\varphi^{2}}\right).
\end{equation}


The manifold $\M^{\varepsilon}$ inherits a natural CR structure from
the complex  structure of the ambient manifold with
$\bH^\varepsilon={\C}T\M^{\varepsilon}\cap T^{1,0}{\mathcal{U}}$.
For a defining function $\varphi$, we define a one--form
\begin{equation}\label{Theta}
\Theta= \frac{i}{2}(\bar{\partial}-\partial)\varphi
\end{equation}
and using the natural embedding map $i_{\varepsilon}:
\M^{\varepsilon}\rightarrow {\mathcal{U}}$, set
$\theta_{\varepsilon}=i_{\varepsilon}^*\Theta$. The contact form
$\theta_{\varepsilon}$ is a pseudohermitian structure for
$\M^{\varepsilon}$. Note that
\begin{equation}\label{dtheta}d\Theta=i\partial\bar\partial\varphi,\end{equation}
and the Levi form on $\mathcal M^\epsilon$ is given by
\begin{equation}\label{Levi0}L_{\theta_{\varepsilon}}=-i d\theta_{\varepsilon}.\end{equation}
We will assume that $L_{\theta_\varepsilon}$ is positive definite
for all $\varepsilon>0$ sufficiently small, so that
$\M^{\varepsilon}$ is strictly pseudoconvex. Moreover, in order to
simplify the notation we will write $\theta$ for
$\theta_{\varepsilon}$, suppressing the $\varepsilon$.

 It is known that
there exists a unique $(1,0)$-vector field $\xi$ on $\mathcal{U}$ such that
$$\partial\varphi (\xi)=1\quad\text{and}\quad \xi\rfloor \partial\bar{\partial}\varphi=r\bar{\partial}\varphi$$
for some $r\in \mathcal C^{\infty}({\mathcal{U}})$.
The function $r$ is called the {\it transverse curvature}.
We decompose $$\xi=\frac{1}{2}(N-iT),$$
 where $N,T$ are real vector fields on $\mathcal{U}$.
Then
\[
d\varphi(N)=2,\quad \theta(N)=0; \quad \theta(T)=1,\quad T\rfloor d\theta=0.
\]
Thus $T$ is the characteristic vector field for each
$M^{\varepsilon}$ and $N$ is normal to $M^{\varepsilon}$.

Let $\{W_{\alpha}\}_{\alpha=1}^n$ be a frame for $\bH^\varepsilon$.
Then, setting $W_ {\bar{\alpha}}=\overline{W}_\alpha$, the vector
fields $\{W_{\alpha},W_{\bar{\alpha}},T\}_{\alpha=1}^n$ form a local
frame for ${\C}T\M^{\varepsilon}$, while
$\{W_{\alpha},W_{\bar{\alpha}}\}_{\alpha=1}^{n+1}$ form a local
frame for ${\C}T{\mathcal{U}}$ by setting
$W_{m}=\xi,W_{\bar{m}}=\bar{\xi}$. If
$\{\theta^{\alpha}\}_{\alpha=1}^n$ is a dual coframe for
$\{W_{\alpha}\}_{\alpha=1}^n$ then, setting $\theta_
{\bar{\alpha}}=\overline{\theta}_\alpha$, the one--forms
$\{\theta^{\alpha},\theta^{\bar{\alpha}},\theta\}$ form a dual
coframe for ${\C}T\M^{\varepsilon}$, while
$\{\theta^{\alpha},\theta^{\bar{\alpha}}\}_{\alpha=1}^{n+1}$ is a
dual coframe for ${\C}T{\mathcal{U}}$ if we denote
$\theta^{m}=\partial \varphi, \theta^{\bar{m}}= \bar{\partial}
\varphi$. The Levi form on each $\mathcal M^\varepsilon$ is given by
$$L_\theta=h_{\alpha\bar\beta}\theta^\alpha
\wedge\theta^{\bar\beta}$$ for a $n\times n$ Hermitian matrix valued
function $h_{\alpha\bar\beta}$.

Recalling \eqref{dtheta}, we can divide $d\Theta$ into tangential and transverse components as follows:
  \[
  \partial\bar{\partial}\varphi=h_{\alpha\bar{\beta}}\theta^{\alpha}\wedge\theta^{\bar{\beta}}
  +r\partial\varphi\wedge\bar{\partial}\varphi,
  \]
which gives the following expression for the  K\"ahler form from formula \eqref{Kahler-form0}
  \begin{equation}\label{Kahler-form}
  \omega_+=\frac{i}{2}\left(\frac{h_{\alpha\bar{\beta}}\theta^{\alpha}\wedge \theta^{\bar{\beta}}}{-\varphi}+
  \frac{(1-r\varphi)}{\varphi^{2}}\partial\varphi\wedge\bar{\partial}\varphi\right).
\end{equation}
For a complex manifold with a Hermitian metric $g=h-i\omega$, the K\"ahler form $\omega$ combines the metric and the complex structure by $h(V_{1},V_{2})=\omega(V_{1},JV_{2})$.  Here the Hermitian-Bergman metric may be written as
\[
g^+= \frac{1}{-\varphi}h_{\alpha\bar{\beta}}\theta^{\alpha}\otimes \theta^{\bar{\beta}}
+ \frac{(1-r\varphi)}{\varphi^{2}}\partial\varphi \otimes\bar{\partial}\varphi.
\]
It is easy to see that
\[
g^+(N,N)=4\frac{(1-r\varphi)}{\varphi^{2}}
\]
and the outward unit normal to $\M^{\varepsilon}$ is
\[
\nu=\frac{-\varphi}{2\sqrt{1-r\varphi}}N.
\]

If $\omega=\frac{i}{2}g_{k\bar l}\, \omega^k\wedge\omega^{\bar l}$ is the expression of a K\"ahler form on a complex manifold in terms of a coframe $\{\omega^k\}$ for $T^{1,0}$, we define the trace of a $(1,1)$-form by
$$\text{Tr}\lp i\eta_{k\bar l} \omega^k\wedge\omega^{\bar l}\rp=g^{k\bar l}\eta_{k\bar l}.$$
The Laplace operator on the K\"{a}hler manifold $(X,\omega_+)$ is
calculated as
 $$\Delta_{+}u =\text{Tr}(i\partial\bar{\partial}u);$$
note that we are using a different normalization constant.

We recall the following formula from \cite{GrahamLee}, that decomposes the Laplacian
into tangential and normal components relative to the level sets of $\varphi$:
\begin{equation}\label{Laplacian1}
\Delta_{+}=\frac{-\varphi}{4}\left[\frac{-\varphi}{1-r\varphi}(N^{2}+T^{2}+2rN+2X_r)+2\Delta_{b}+2nN\right],
\end{equation}
where $\Delta_b$ is given by \eqref{sublaplacian} and $r$ is the
transverse curvature. For any real function $f$, we have defined the
real vector field $X_f$, analogous to the gradient of $f$ in
Riemannian geometry, by
$$X_f=f^\alpha W_\alpha+f^{\bar\beta}W_{\bar\beta}.$$ The proof in \cite{GrahamLee} also shows
that $1-r\varphi>0$.

\subsection{Approximate K\"ahler-Einstein and asymptotically complex hyperbolic manifolds}\label{section-ACH}

Let $\Omega$ be a bounded strictly pseudoconvex domain in $\mathbb
C^m$. Fefferman \cite{Fefferman:Monge-Ampere} showed the existence
of a defining function $\varphi$ for $\partial\Omega$ which is an
approximate solution of the complex Monge-Amp\`ere equation
\begin{equation*}\label{complexMA}
\left\{\begin{split}
&J[\varphi]=1, \\
&\varphi|_{\partial\Omega}=0,
\end{split}\right.
\end{equation*}
where $J[\varphi]$ is the Monge-Amp\`ere operator
\begin{equation*}J[\varphi]=(-1)^m \det
\lp\begin{array}{cc}
\varphi & \varphi_{\bar j} \\
\varphi_i & \varphi_{i\bar j}
\end{array}\rp.\end{equation*}

The K\"ahler metric $g^+$ defined from the K\"ahler form \eqref{Kahler-form}
  associated to such an approximate solution $\varphi$ is an approximate K\"ahler-Einstein metric on $\Omega$, i.e. $g^+$ obeys
\begin{equation}\label{Einstein-equation}\Ric (g^+)=-(m+1)\omega_+ +\mathcal O(\varphi^{m-1}),\end{equation}
where $\Ric$ is the Ricci form.

Under certain conditions (see
\cite{Cheng-Yau:Kahler,Lee-Melrose,HislopPeterTang}),  Fefferman's
local approximate solution of the Monge-Amp\`ere equation can be
globalized to an approximate solution of the Monge-Amp\`ere equation
near the boundary of a complex manifold $\X$ with strictly
pseudoconvex boundary $\mathcal M$. It follows that $\X$ carries an
approximate K\"ahler-Einstein metric in the sense of
\eqref{Einstein-equation}. This $\varphi$ is called a globally
defined approximate solution of the Monge-Amp\`ere equation in $\X$.

More precisely, such a solution exists if and only if $\mathcal M$
admits a pseudohermitian structure $\theta$ which is
volume-normalized with respect to  some locally defined, closed
$(n+1,0)$-form in a neighborhood of any point $p\in\M$. If
$\dim(\M)\geq 5$, this condition is equivalent to the existence of a
pseudo-Einstein, pseudohermitian structure $\theta$ on $\M$ (see
\cite{Lee:pseudoeinstein}). Recall that a CR manifold is
pseudo-Einstein if there is a pseudohermitian structure $\theta$ for
which the Webster-Ricci curvature is a multiple of the Levi form.
Note that $\dim(\M)=3$ is special. Finally, in the particular case
that $\X$ is a pseudoconvex domain in $\mathbb C^m$, this condition
is trivially satisfied.

\vspace{0,1cm}

We describe now a more general class of manifolds than the ones
considered so far. In particular, we consider the
$\Theta$-structures introduced by Epstein-Melrose-Mendoza
\cite{Epstein91}. We will only give brief description and we refer
the interested reader to the detailed description in
\cite{Guillarmou08}.

Let $\X$ be a non-compact manifold of (real) dimension $2n+2$ with a
Riemannian metric $g^+$ that compactifies into a smooth $\bar\X$,
with boundary $\partial \X$. We assume that the boundary admits a
contact form $\theta$ and an almost complex structure $J:\text{Ker}
(\theta) \to\text{Ker} (\theta)$ such that $d\theta(\cdot,J\cdot)$
is symmetric and positive definite on $\text{Ker} (\theta)$. The associated
characteristic direction $T$ is characterized by
$$\theta(T)=1,\quad d\theta(T,JZ)=0\, \mbox{ for any }\, Z\in\text{Ker}
(\theta),$$ see \eqref{dm3}, \eqref{Webstermetric-dm}.

We say that $(\X,g^+)$ is an asymptotically complex hyperbolic
manifold if there exists a diffeomorphism $\phi:[0,\epsilon)\times
\partial \bar\X\to \X$ such that $\phi(\{0\}\times
\partial\bar\X)=\partial \bar \X$ and such that the metric splits as a
product of the form
$$\phi^* g^+=\frac{4d\rho ^2+d\theta(\cdot,J\cdot)}{\rho ^2}+\frac{\theta^2}{\rho^{4}}+\rho Q_\rho=:\frac{4d\rho^2+h(\rho)}{\rho^2},$$
for some smooth symmetric tensors $Q_\rho$ on $\partial \bar\X$,
where $\rho$ is a defining function for $\partial\bar \X$.

Note that if $\rho$ is any boundary defining function, then $\rho^4
g^+|_{\partial \bar \X}= e^{4w}\theta^2$ for some smooth function
$w$ on $\partial\bar\X$. Thus it is natural to define the pair
$([\theta],J)$ a conformal pseudohermitian structure on
$\partial\bar\X$.

We say that $g$ is even at order $2k$ if $h^{-1}(\rho)$ has only
even powers in its Taylor expansion at $\rho=0$ at order $2k$, where
$h^{-1}(\rho)$ is the metric dual to $h(\rho)$.

\section{Model case: the Heisenberg group}\label{section-Heisenberg}

The real $n$-dimensional Heisenberg group $\mathbb H^n$ is defined as the set $\R^n\times\R^n\times\R$ endowed with the
group law
$$(\hat x,\hat y,\hat t)\circ(x,y,t)=\big(\hat x+x,\hat y+y,\hat t+t+2(\langle x,\hat
y\rangle_{\R^n}-\langle \hat x, y\rangle_{\R^n})\big)$$ for $(\hat
x,\hat y,\hat t),\,(x,y,t)\in\mathbb H^n$, where
$\langle\cdot,\cdot\rangle_{\R^n}$ is the standard inner product in
$\R^n$. Alternatively, we can use complex coordinates $z=x+iy$ to
denote elements of $\R^n\times\R^n\simeq\C^n$, so that the group
action in $\mathbb{H}^n$ can be written as
$$(\hat z,\hat t)\circ(z,t)=(\hat{z}+z,\hat{t}+t+2\text{Im}\langle \hat z,z\rangle_{\C^n})$$ for
$(\hat{z},\hat t),\,(z,t)\in\mathbb H^n$, where and
$\langle\cdot,\cdot\rangle_{\C^n}$ is the standard inner product in
$\C^n$.

For any fixed $(\hat{z},\hat t)\in\mathbb{H}^n$, we will denote
$\tau_{(\hat{z},\hat t)}:\mathbb{H}^n\rightarrow\mathbb{H}^n$ the
left translation on $\mathbb{H}^n$ by $(\hat{z},\hat t)$ defined by
$$\tau_{(\hat{z},\hat t)}(z,t)\,\,=\,\,(\hat{z},\hat t)\circ(z,t),\qquad\forall\,(z,t)\in\mathbb{H}^n.$$ For any $\lambda>0$ the
dilation $\delta_\lambda:\mathbb{H}^n\rightarrow\mathbb{H}^n$ is
$$\delta_\lambda(z,t)\,\,=\,\,(\lambda
z,\lambda^2t)\,\,=\,\,(\lambda x,\lambda y,
\lambda^2t),\qquad\forall\,(z,t)=(x,y,t)\in\mathbb{H}^n.$$ Notice
that
$$\delta_\lambda\big((\hat{z},\hat t)\circ(z,t)\big)=\delta_\lambda\big(\hat z,\hat t)\circ\delta_\lambda\big(z,t)$$
for every $\lambda>0$ and $(\hat{z},\hat t),\,(z,t)\in\mathbb H^n$.


As we will see, the $n$-dimensional Heisenberg group can be regarded
as  a smooth sub--Riemannian manifold. An orthonormal frame on the
manifold is given by the Lie vector fields
\begin{equation}\label{vectorfields}
X_j=\frac{\partial}{\partial x_j}+2y_j\frac{\partial}{\partial
t},\quad Y_j=\frac{\partial}{\partial
y_j}-2x_j\frac{\partial}{\partial t}\quad\text{for
}j=1,\ldots,n,\quad T_0=\frac{\partial}{\partial t},
\end{equation}
which form a base of the Lie algebra of vector fields on the
Heisenberg group which are left invariant with respect to the group
action $\circ$. Notice that
$$[X_j,Y_k]=-4T_0\delta_{jk},\qquad[X_j,X_k]=[Y_j,Y_k]=0\qquad\text{for }j,k=1,\ldots,n,$$ where $\delta_{jk}$ is the Kronecker symbol. We also set
\begin{align}\label{3}
Z_j\,=\,\frac{\partial}{\partial
z_j}+i\bar{z}_j\frac{\partial}{\partial t}, \quad \bar
Z_j\,=\,\frac{\partial}{\partial \bar
z_j}-i{z}_j\frac{\partial}{\partial t},
\end{align}
 where
\begin{align}
\frac{\partial}{\partial z_j}=\frac{1}{2}\left(\frac{\partial}{\partial x_j}-i\frac{\partial}{\partial y_j}\right), \quad \frac{\partial}{\partial \bar z_j}=\frac{1}{2}\left(\frac{\partial}{\partial x_j}+i\frac{\partial}{\partial y_j}\right),
\end{align}
for $j=1,\ldots,n$.

\subsection{The Heisenberg group as a CR manifold}

The CR structure on the Heisenberg group $\cH^{n}$ is given by the
$n$-dimensional complex bundle
\begin{equation}\label{9}
\mathcal{H}(\cH^{n})=\text{span}_\C\langle Z_1,\ldots,Z_n\rangle
\end{equation}
with $Z_1,\ldots,Z_n$ given by \eqref{3}. We immediately observe
that the associated maximal complex distribution
$H(\cH^n)=\text{Re}\{\mathcal{H}(\cH^{n})\oplus
\overline{\mathcal{H}}(\cH^{n})\}$ is simply
$$H(\mathbb{H}^n)=\text{span}_\R\langle
X_1,\ldots,X_n,Y_1,\ldots,Y_n\rangle,$$ and that it carries the
complex structure $J_b:H(\mathbb{H}^n)\rightarrow H(\mathbb{H}^n)$
defined by
$$J_b X_j=Y_j,\qquad J_b Y_j=-X_j\qquad\text{for }j=1,\ldots,n.$$
The associated one--form $\theta_0$ satisfying \eqref{dm1}, which is
globally defined on $\cH^n$, is precisely
\begin{equation}\label{theta0}
\theta_0\,\,=\,\,dt+i\sum_{j=1}^n\left(z_jd\bar{z}_j-\bar{z}_jdz_j\right)\,\,
=\,\,dt+2\sum_{j=1}^n\left(x_jdy_j-y_jdx_j\right).
\end{equation}
From here  one immediately calculates
$$d\theta_0\,\,=\,\,2i\sum_{j=1}^ndz_j\wedge d\bar{z}_j\,\,=\,\,4\sum_{j=1}^ndx_j\wedge dy_j.$$
Then, the characteristic direction  defined through \eqref{dm3} is simply the vector field $T_0$ previously defined. The real tangent bundle of $\mathbb H^n$ satisfies
 $$T\cH^n=H(\cH^n)\oplus\R T_{0}.$$
The Levi form $L_{\theta_{0}}$ associated to the pseudohermitian
structure $\theta_{0}$ on $\cH^{n}$ constructed from \eqref{dm2a} is
given by
$$L_{\theta_{0}}(Z_j,\bar{Z}_k)=2\delta_{jk}\qquad\text{ for
}j,k=1,\ldots,n,$$
which is a positive definite matrix. This tells us that $(\mathbb{H}^n,\theta_0)$ is strictly
pseudoconvex.
Moreover, one can define the Webster metric $g_{\theta_0}$ on
$\cH^n$ by the relations \eqref{Webstermetric-dm}; in particular,
\begin{equation*}\begin{array}{ccc}
\label{4}&g_{\theta_0}(X_j,X_k)=g_{\theta_0}(Y_j,Y_k)=4\delta_{jk},\qquad
g_{\theta_0}(T_{0},T_{0})=1,&\\
&g_{\theta_0}(X_j,Y_k)=g_{\theta_0}(X_j,T_{0})=g_{\theta_0}(Y_j,T_{0})=0.&
\end{array}\end{equation*}
A volume form on $\cH^n$ is given by
\begin{eqnarray*}
\psi_{0}\,\,=\,\,\theta_0\wedge(d\theta_{0})^n&=&n!(2i)^{n}dz_1\wedge
d\bar{z}_1\wedge\ldots\wedge dz_n\wedge d\bar{z}_n\wedge
dt\\
&=&n!2^{2n}dx_1\wedge dy_1\wedge\ldots\wedge dx_n\wedge dy_n\wedge
dt.
\end{eqnarray*}

Now, since $X_1,\ldots X_n,Y_1,\ldots,Y_n,T_0$ is a basis of
$T\cH^n=H(\cH^n)\otimes\R T_{0}$, we have from \eqref{dm7} that
\[
\nabla_{g_{\theta_0}}u=T_{0}(u)T_{0}+\frac{1}{4}\sum_{j=1}^nX_j(u)X_j+Y_j(u)Y_j
\]
for any smooth function $u$ on $\mathbb{H}^n$. We observe that
$$du=T_{0}u\,\theta_0+\sum_{j=1}^nX_ju\,dx_j+Y_ju\,dy_j,$$ so $$d_bu=\sum_{j=1}^nX_ju\,dx_j+Y_ju\,dy_j.$$
The horizontal gradient is calculated from \eqref{dm8}
\begin{equation*}
\nabla_{b}u:=\frac{1}{4}\sum_{j=1}^nX_j(u)X_j+Y_j(u)Y_j,
\end{equation*}
and the sub--Laplacian associated to $\theta_0$ from \eqref{dm5}
\begin{align*}
\tilde\Delta_{b}u:=\frac{1}{4}\sum_{j=1}^n\left[X^2_j+Y^2_j\right]u=\frac{1}{2}\sum_{j=1}^n
\left[Z_j \bar Z_j+\bar Z_j Z_j\right] u;
\end{align*}
the differential operator is linear, second order, degenerate
elliptic, and it is hypoelliptic being the sum of squares of
(smooth) vector fields satisfying the H\"{o}rmander condition
\cite{Hormander}.

Assume that $u,v$ are smooth functions on $\mathbb{H}^n$ and at
least one of them is compactly supported, then we see that
$$-\int_{\cH^n}v\tilde\Delta_bu\,\psi\,\,=\,\,\int_{\cH^n}g_{\theta_0}(\nabla_{b}v,\nabla_{b}u)\,\psi.$$
The sub--Laplacian on the Heisenberg group or on a orientable,
strictly pseudoconvex CR manifold of hypersurface type is the
Laplace-Beltrami operator on the manifold, which plays an important
role in harmonic analysis and partial differential equations.

For the rest of the paper, we will be working with the new contact
form $\theta=\theta_0/2$ on the Heisenberg group, since $\theta$ is
the CR structure inherited from the complex hyperbolic space (see
\eqref{new-theta}). Therefore, our sub--Laplacian on the Heisenberg
group, by definition will be
\begin{align}
\Delta_{b}u:=\frac{1}{2}\sum_{j=1}^n\left[X^2_j+Y^2_j\right]u
\end{align}
and the characteristic direction will be $T=2T_0$. This explains the
choice of constants in \eqref{Laplacian0}.

\bigskip

\subsection{The Heisenberg group as the boundary of the Siegel domain}

The Heisenberg group may be identified with the boundary of a domain in $\mathbb C^{n+1}$. Indeed, let $\Omega_{n+1}$ be the Siegel domain, defined by
$$\Omega_{n+1}:=\left\{\zeta=(z_1,\ldots,z_n,z_{n+1})=(z,z_{n+1})\in\C^{n}\times\C
\,|\,q(\zeta)>0\right\}.$$
where
$$q(\zeta)=\text{Im}\,z_{n+1}-\sum_{j=1}^n|z_j|^2.$$
Its boundary
$\partial\Omega_{n+1}=\left\{\zeta\in\C^{n+1}\,|\,q(\zeta)=0\right\}$
is an oriented CR manifold of hypersurface type with the CR
structure induced by $\C^{n+1}$. Now define  $G:\cH^n\rightarrow\partial\Omega_{n+1}$,
\begin{align}G(z,t)&=(z,t+i|z|^2),\,\,(z,t)\in\cH^n,\\
G^{-1}(\zeta)&=(z_1,\ldots,z_n,\text{Re}\,z_{n+1}),\,\,\zeta\in\partial\Omega_{n+1}.
\end{align}
One may check that $G$ is a CR isomorphism between $\cH^n$ and
$\partial\Omega_{n+1}$, i.e. it preserves the CR structures.

\vspace{0,1cm}

The boundary manifold inherits a natural CR structure from the
complex structure of the ambient manifold, using  $\varphi=-q$ as a
defining function. The pseudohermitian structure on $\cH^n$ is then
given (via pullback) by the contact form
\begin{equation}\label{new-theta}
\theta=\frac{i}{2}(\bar\partial-\partial)\varphi=
\frac{1}{4}(dz_{m}+d\bar{z}_{m})+
 \frac{i}{2}\sum_{\alpha=1}^{n}z_{\alpha}d\bar{z}_{\alpha}-\bar{z}_{\alpha}dz_{\alpha}=\theta_{0}/2
\end{equation}
where $\theta_0$ was defined in \eqref{theta0}.

We now set up the orthonormal frame we will be using for the rest of
the paper. The subindexes $\alpha,\beta$ will run from 1 to $n$. As
in formula \eqref{3} let
$$Z_{\alpha}=\frac{\partial }{\partial z^{\alpha}}+i\bar{z}_{\alpha}\frac{\partial }{\partial
t},\,\,\,\,\, Z_{\bar{\alpha}}=\frac{\partial }{\partial
\bar{z}^{\alpha}}-iz_{\alpha}\frac{\partial }{\partial
t}\qquad\alpha=1,\ldots,n.$$
In order to complete a basis for $T\Omega_{n+1}$, define the two
real vector fields
$$T=2\frac{\partial }{\partial t}\quad\mbox{and}\quad N= -2\frac{\partial }{\partial q}.$$
Note that $T$ differs from the $T_0$ from the previous section by a factor of 2 in order to match with the contact form \eqref{new-theta}.
Define also
 \[
 \xi_{m}=\frac{1}{2}(N-iT)\quad\text{and}\quad \bar{\xi}_{m}=\frac{1}{2}(N+iT).
 \]
Then a  frame in $\Omega_{n+1}$ is given by
\begin{equation}\label{frame}\{Z_{\alpha}, Z_{\bar{\alpha}},\xi_{m},\bar{\xi}_{m} \}.
\end{equation}
The dual coframe  in $\C T\mathbb H^{n}$ of $\{Z_{\alpha}, Z_{\bar{\alpha}},T \} $ is given by
 $\{\theta^{\alpha},\theta^{\bar{\alpha}},\theta\}$, while the dual coframe of the basis \eqref{frame} in $\Omega_{n+1}$ is simply
 $\{\theta^{\alpha},\theta^{\bar{\alpha}},\theta,dq\}$, where $dq=\partial q+\bar{\partial} q$.
In particular,
 $T\rfloor\theta=1$ and  $Z_{\alpha}\rfloor\theta^{\alpha}=1$.

 Finally, one may calculate the corresponding Levi form from \eqref{Levi0}
 \begin{align}\label{partialbarpartial}
 L_\theta=  {\partial}\bar\partial \varphi(z)=  \sum_{\alpha=1}^{n}dz_{\alpha}\wedge d\bar{z}_{\alpha},
  \end{align}
which may be rewritten as
$$L_\theta=h_{\alpha\bar{\beta}}\theta^{\alpha}\wedge\theta^{\bar{\beta}}\quad
\mbox{for}\quad h_{\alpha\bar\beta}=\frac{\partial^{2}
\varphi}{\partial z_{\alpha}\partial
z_{\bar{\beta}}}=\delta_{\alpha\bar\beta}.$$ Now one may calculate
the sub--Laplacian with respect to the contact form $\theta$.
Indeed,
$$\Delta_b u= \sum_{\alpha,\beta=1}^n h^{\alpha \bar\beta} \left[ u_{\alpha\bar\beta}+u_{\bar\beta\alpha}\right]=\frac{1}{2}\sum_{j=1}^n \left[X_j^2+Y_j^2\right]u,$$
which differs from the sub--Laplacian associated to $\theta_0$ by a
factor of 2.

On the other hand, let us look at the complex structure of
$\Omega_{n+1}$. The functions $t=\text{Re}\, z_m\in\mathbb R$,
$z_1,\ldots,z_n\in\mathbb C^n$ and $\rho=(2q(z))^{1/2}\in(0,\infty)$
give coordinates in the Siegel domain $\Omega_{n+1}\simeq\mathbb
H^{n}\times (0,\infty)$. For the defining function $-\varphi$, one
can construct a K\"ahler form in $\Omega_{n+1}$ as
 \begin{align}\label{kahlerform}
 \omega_{+}&=-\frac{i}{2}{\partial}\bar\partial \log(\varphi)\,=\, \frac{i}{2}\left(\frac{{\partial}\bar\partial \varphi}{-\varphi}+\frac{{\partial}\varphi\wedge\bar \partial \varphi}{\varphi^{2}}\right).
   \end{align}
The first term $\partial\bar\partial \varphi$ is calculated in \eqref{partialbarpartial}, while for the second we observe that
$$\partial\varphi\wedge\bar\partial\varphi=\frac{1}{4}\left[(\partial \varphi+\bar{\partial}\varphi)^{2}-(\partial \varphi-\bar{\partial}\varphi)^{2}\right]
=\frac{1}{4}\left[d\varphi^ 2+4\theta^2\right],$$
where we have used \eqref{Theta} in the last step. Then the K\"ahler form is simply
 \begin{align}\label{Kahler1}
 \omega_{+}&=\frac{i}{2}\left(\frac{h_{\alpha \bar{\beta}}\theta^\alpha\wedge\theta^{\bar \beta}}{-\varphi}+\frac{\tfrac{1}{4}d\varphi^2+\theta^{2}}{\varphi^2}\right),
\end{align}
After the change $q=\rho^2/2$, the Hermitian-Bergman metric is given by
\begin{equation}\label{Kahler-metric}
g^{+}=\frac{1}{2}\left(\frac{d\rho^{2}}{\rho^{2}}+\frac{2\delta_{\alpha\bar{\beta}}\theta^{\alpha}\otimes \theta^{\bar{\beta}}}{\rho^{2}}+
\frac{4\theta^{2}}{\rho^{4}}
\right).\end{equation}
Then $(\Omega_{n+1},\omega_+)$ is a K\"ahler manifold with constant holomorphic curvature.

\subsection{Cayley transform}

The Heisenberg group is also CR isomorphic to the sphere
$\mathbb{S}^{2n+1}\subset\C^{n+1}$ minus a point. The CR equivalence
between the two manifolds is given by the map
$\Psi_c:\mathbb{S}^{2n+1}\setminus\{(0,\ldots,0,-1)\}\rightarrow\cH^n$,
defined by
$$\Psi_{c}(\zeta)=\left(\frac{\zeta_1}{1+\zeta_{n+1}},\ldots,\frac{\zeta_n}{1+\zeta_{n+1}},
\text{Re}\Big(i\frac{1-\zeta_{n+1}}{1+\zeta_{n+1}}\Big)\right),$$
for
$\zeta\in\mathbb{S}^{2n+1}\setminus\{(0,\ldots,0,-1)\}\subset\C^{n+1}$.
Notice that
$$\Psi_{c}^{-1}(z,t)=\left(\frac{2iz}{t+i(1+|z|^2)},\frac{-t+i(1-|z|^2)}{t+i(1+|z|^2)}\right),
\,\,\,\,(z,t)\in\cH^n.$$
The Jacobian determinant of this transformation is
\begin{equation}\label{Jacobian}
J_{c}(z,t)=\frac{2^{2n+1}}{((1+|z|^{2})^{2}+t^{2})^{n+1}}
\end{equation}
so that
\[
\int_{{\mathbb{S}}^{2n+1}}f\,dS=\int_{\cH^n}(f\circ\Psi^{-1}_{c})|J_{c}|\,dzdt.
\]

The CR unitary sphere $\mathbb{S}^{2n+1}$ is the boundary of the
ball model for the complex hyperbolic space of complex dimension
$n+1$, which is the unit ball $B^{n+1}=\{\zeta\in\mathbb C^{n+1} :
|\zeta|<1\}$ equipped with the K\"ahler metric
$g_0=-4\partial\bar\partial \log r$, where $r=1-|\zeta|^2$. The
holomorphic curvature is constant (and negative) and this metric is
called the Bergman metric.

\section{Scattering theory and the conformal fractional sub--Laplacian}

Now we go back to the general setting described in section
\ref{section-ACH}. Let $\X$ be a $m$-dimensional complex manifold
with compact, strictly pseudoconvex boundary $\partial\X=\mathcal M$. Let
$g^+$ be a K\"ahler metric on $\X$ such that there exists a globally
defined approximate solution $\varphi$ of the Monge-Amp\`ere
equation that makes $g^+$ an approximate K\"ahler-Einstein metric in
the sense of \eqref{Einstein-equation} and $\mathcal
X=\{\varphi<0\}$. In particular, the metric $g^+$ belongs to the
class of $\Theta$-metrics considered by \cite{Epstein91}.

The spectrum of the Laplacian $-\Delta_+$ in the metric $g^+$ consists of an absolutely continuous part
$$\sigma_{{ac}}(-\Delta_+)=\left[\tfrac{m^2}{4},\infty\right),$$
and the pure point spectrum satisfying
$$\sigma_{{pp}}(-\Delta_+)\subset\left(0,\tfrac{m^2}{4}\right),$$
and moreover it consists of a finite set of $L^2$-eigenvalues. The
main result in \cite{Epstein91} is the study of the modified
resolvent
$$R(s)=(-\Delta_+ -s(m-s))^{-1}$$
considered in $L^2(\X)$, which allows to define the Poisson map. More precisely, let
$$\Sigma:=\{s\in\mathbb C \,:\,  \Re(s)>m/2,\, s(m-s)\in\sigma_{pp}(-\Delta_+)\};$$
 the resolvent operator is meromorphic for $\Re(s)>\frac{m}{2}-\frac{1}{2}$, having at most finitely many, finite-rank poles at $s\in\Sigma$. Moreover, for $s\not\in\Sigma$ and $\Re(s)>\frac{m}{2}-\frac{1}{2}$,
$$R(s):\dot{\mathcal C}^\infty(\X) \to q^s\mathcal C^\infty(\X),$$
where  $\dot{\mathcal C}^\infty(\X)$ is the set of $\mathcal C^\infty$ functions on $\mathcal X$ vanishing up to infinite order at $\partial \mathcal X$.

The Poisson operator is constructed as follows (see
\cite{HislopPeterTang}). Let $s\in\mathbb C$ such that
$\Re(s)\geq\frac{m}{2}$, $s\not\in \mathbb Z$ and
$2s-m\not\in\mathbb Z$. Let $q=-\varphi$. Then, given $f\in\mathcal
C^{\infty}(\partial\X)$, there exists a unique solution $u_s$ of the
eigenvalue problem
\begin{equation*}
-\Delta_+ u_s -s(m-s)u_s=0,
\end{equation*}
with the expansion
\begin{equation*}\left\{\begin{split}
&u_s=q^{(m-s)}F+q^{s} G, \quad \mbox{for some}\quad F,G\in\mathcal C^\infty(\bar{\X}),\\
& \,F|_{\mathcal M}=f.
\end{split}\right.\end{equation*}
Then the Poisson map is defined as $\mathcal P_s:\mathcal C^\infty(\mathcal M)\to\mathcal C^\infty(\mathring{\mathcal X})$ by $f\mapsto u_s$, and the scattering operator
$$S(s): \mathcal C^\infty(\mathcal M) \to \mathcal C^\infty(\mathcal M)$$
by
$$S(s)f:=G|_{\mathcal M}.$$
Note that $S(s)$ is a meromorphic family of pseudodifferential operators in the $\Theta$-calculus of \cite{Epstein91} of order $2(2s-m)$, and self-adjoint when $s$ is real.

The scattering operator has infinite-rank poles when $\Re(s)>\frac{m}{2}$ and $2s-m\in\mathbb Z$
due to the crossing of indicial roots for the normal operator.
At those exceptional points $s_k=\frac{m}{2}+\frac{k}{2}$, $k\in\mathbb N$, $s_k\not \in\Sigma$, the CR operators may be recovered by calculating the corresponding residue
$$\underset{{s=s_k}}{\Res} \,S(s)=p_k,$$
where $p_k$ is a CR-covariant differential operator of order $2k$, with principal symbol
$$p_k=\frac{(-1)^k}{2^k k!(k-1)!} \prod_{l=1}^k (-\Delta_b+i(k+1-2l)T) \quad + \quad \mbox{l.o.t.}$$
These correspond precisely to the GJMS operators \cite{Graham-Jenne-Mason-Sparling} as constructed by Gover-Graham in \cite{Gover-Graham:CR-powers}.

For $\gamma\in(0,m)\backslash\mathbb N$, set $s=\frac{m+\gamma}{2}$.
We define the CR fractional sub--Laplacian on $(\mathcal M,\theta)$
by
$$P_\gamma^\theta f=c_\gamma S(s)f,$$
for a constant
$$c_\gamma=2^\gamma\frac{\Gamma(\gamma)}{\Gamma(-\gamma)}.$$
$P_\gamma^\theta$ is a pseudodifferential operator of order
$2\gamma$ with principal symbol given by \eqref{symbol}.

The main property of the operator is its CR covariance. Indeed, if $\hat\varphi=v^{\frac{2}{m-\gamma}}\varphi$ is another defining function for $\mathcal M$ and $v|_{\mathcal M}=w$, which gives a relation between the contact forms as $\hat \theta=w^{\frac{2}{m-\gamma}}\theta$, then the corresponding operator is given by
$$P_\gamma^{\hat\theta}(\cdot)=w^{-\frac{m+\gamma}{m-\gamma}}P_\gamma^\theta(w\,\cdot).$$
In particular, for $\gamma=1$, we obtain the CR Yamabe operator of Jerison-Lee \cite{Jerison-Lee:Yamabe}
$$P_1^\theta=-\Delta_b+\frac{n}{2(n+1)}R_\theta,$$
where $\Delta_b$ is the sub--Laplacian on $(\M,\theta)$ and
$R_\theta$ is the associated Webster curvature.

The Heisenberg group $(\mathbb H,\theta)$, for $\theta$ given by \eqref{new-theta}, is the model case for $\mathcal M$, when it is understood as the boundary at infinity of the complex Poincar\'e ball through the Cayley transform. In particular, $P^\theta_\gamma$ agrees with the intertwining operators on the CR sphere calculated in \cite{Branson-Olafsson-Orsted,Graham:compatibility,BrFoMo}. Moreover, we can write explicitly
$$P_1^\theta=-\Delta_b, \quad P_2^\theta= (\Delta_b)^2+T^2,\quad
P_k^\theta=\prod_{l=1}^k (-\Delta_b+i(k+1-2l)T).$$

In the case that $\X$ is an asymptotically hyperbolic manifold as
described in \cite{Guillarmou08}, not necessarily coming from an
approximate solution of the Monge-Amp\`ere equation, one can still
construct the scattering operator and the CR fractional
sub--Laplacian, except possibly for additional poles at the values
$\gamma=\frac{k}{2}$, $k\in\mathbb N$. In \cite{Guillarmou08}, a
careful study of those additional values is obtained with the
assumption that the metric is even up to a high enough order.

\vspace{0,1cm}

One may define also the Branson's fractional CR curvature as
$$Q_\gamma^\theta= P^\theta_\gamma(1), \quad \gamma\in(0,m),$$
with a constant in front that we assume to be 1.

In the critical case $\gamma=m$, the operator $P^\theta_m$ was first
introduced in \cite{Graham:compatibility} (see also
\cite{Graham:thesisI}, \cite{Graham:thesisII}) as a compatibility
operator for the Dirichlet problem for the Bergman Laplacian for the
ball in $\mathbb C^m$. This construction was later generalized to
the boundary of a strictly pseudoconvex domain in $\mathbb C^m$ in
\cite{GrahamLee}. The CR $Q=Q_m$ curvature may be calculated as
 \begin{equation}\label{Branson-curv}c_m Q=\lim_{s\to m} S(s)1.\end{equation}
 It is a conformal invariant in the sense that, for a change of contact form $\hat\theta=e^{2w}\theta$,
 $$e^{2mw} Q^{\hat \theta}=Q^\theta+P_m w,$$
 as it was shown in \cite{Fefferman-Hirachi}. The $Q$-curvature also appears in the calculation of renormalized volume (see \cite{Seshadri:volume}).


\section{The extension problem for the Heisenberg group}

In this section we give the proof of Theorem \ref{thm1}. Here $(\mathcal X,g^+)$ is the Siegel domain $\Omega_{n+1}$ with the complex hyperbolic metric, and with boundary $\mathcal M=\{\varphi=0\}$ the Heisenberg group. Throughout, we assume that $s> m/2$ and we parametrize $s=(m+\gamma)/2$ with $\gamma>0$. (At the end we will also assume that $\gamma<1$, but this is not needed for most of the discussion.)

First  recall the formula for the calculation of the Laplacian $\Delta_+$ from \eqref{Laplacian1}. Denote $q=-\varphi$, $N=-2\partial_q$ and $T=2\partial_t$. Then we may write
$$\Delta_+=q\left[q(\partial_{qq}+\partial_{tt})+\tfrac{1}{2}\Delta_b-n\partial_q\right],$$ with $\Delta_b$ the sub--Laplacian \eqref{Laplacian0}.

Now consider the scattering equation
\begin{equation}\label{scattering-equation}-\Delta_+ u -s(m-s) u=0.\end{equation}
We are looking for solutions which are small (in a certain sense) as $q\to\infty$ and which behave for $q\to 0$ like
$$
u = q^{m-s} F + q^s G \,.
$$
We are interested in the map $F|_{q=0} \mapsto G|_{q=0}$. To extract the leading term we substitute $u= q^{m-s} U$ into \eqref{scattering-equation} and obtain the new equation
\begin{equation}\label{scattering-equation-new}
\left( q\partial_{qq} + (1-\gamma)\partial_q + q\partial_{tt} + \frac12\Delta_b \right) U = 0 \,.
\end{equation}
One additional change of variables $q=\rho^2/2$ transforms
\eqref{scattering-equation-new} into the more familiar extension
problem \eqref{equation-CafSil}. This is the analogue to the
Caffarelli-Silvestre extension \cite{CaSi} on  the Euclidean space,
with the additional term in the $t$-direction that appears in the
Heisenberg group case.

In addition, one may recover the scattering operator as in \cite{Chang-Gon} by
$$P_\gamma^\theta f=c_\gamma S\lp\frac{m+\gamma}{2}\rp f=  \frac{c_\gamma}{\gamma}\lim_{q\to 0} q^{1-\gamma} \partial_q U=\frac{c_\gamma}{\gamma2^{1-\gamma}}\lim_{\rho\to 0} \rho^a \partial_\rho U. $$

\subsection{The group Fourier transform}

Let us recall the Fourier transform on the Heisenberg group $\mathbb
H^{n}$, defined by using the irreducible representation of $\mathbb
H^{n}$ from the Stone--Von Neumann theorem (see
\cite{Bahouri-Gallagher:Heat,Geller77,Thangavelu:book}
  for the necessary background).
For a holomorphic function $\Psi$ on $\C^n$ let
\[
\pi_{z,t}^{\lambda}\Psi(\xi)=\Psi(\xi-\bar{z})e^{i\lambda t+2\lambda(\xi\cdot z-|z|^{2}/2)},\;\lambda>0,
\]
\[
\pi_{z,t}^{\lambda}\Psi(\xi)=\Psi(\xi+z)e^{i\lambda
t+2\lambda(\xi\cdot \bar{z}-|z|^{2}/2)},\;\lambda<0,
\]
where $\xi\cdot z=\sum_{j=1}^n\xi_j z_j$ for $\xi,z\in\C^n$. It is
easy to check that the unitary family
$\{\pi_{z,t}^{\lambda}\},\lambda\in \R\setminus\{0\}$ satisfies
$\pi_{z,t}\pi_{\hat{z},\hat{t}}=\pi_{(z,t)\circ (\hat{z},\hat{t})}$.
Moreover, it
 gives all irreducible representations of
$\mathbb  H^{n}$ (except those trivial on the center).

Consider also the Bargmann spaces
\[
{\mathcal{G}}_{\lambda}:=\{ \Psi\;\text{holomorphic in}\;
{\mathbb{C}^{n}}, \|\Psi\|_{{\mathcal G}_{\lambda}}<\infty \}
\]
where $$\|\Psi\|^{2}_{{\mathcal{G}}_{\lambda}}:=\left(\frac{2|\lambda|}{\pi}\right)^{n}\int_{{\mathbb{C}}^{n}} |\Psi(\xi)|^{2}e^{-2|\lambda||\xi|^{2}}\,d\xi.$$
The space ${\mathcal{G}}_{\lambda}$ is a Hilbert space with orthogonal basis
\[
\Psi_{\alpha,\lambda}(\xi)=\frac{(\sqrt{2|\lambda|}\xi)^{\alpha}}{\sqrt{\alpha !}} \quad \text{for}\;\alpha\in \mathbb N_0^{n}.
\]
Here we adopt the usual multi-index conventions $\xi^{\alpha}:=\xi_{1}^{\alpha_{1}}\xi_{2}^{\alpha_{2}}\cdots\xi_{n}^{\alpha_{n}}$,
 $|\alpha|= \alpha_{1}+\alpha_{2}+\cdots+\alpha_{n}$, $\alpha!=\alpha_{1}!\alpha_{2}!\cdots\alpha_{n}!$ for $\alpha=(\alpha_{1},\alpha_{2},\cdots,\alpha_{n})$, and $\mathbb N_0=\{0,1,2,\cdots\}$.
For $\displaystyle\varphi=\sum_{\alpha\in{\mathbb{N}}_{0}^n}b_{\alpha}\xi^{\alpha}\in {\mathcal{G}}_{\lambda}$ we have that
\[
\|\varphi\|_{{\mathcal{G}}_{\lambda}}:=\sum_{\alpha\in{\mathbb{N}}_{0}^{n}}\alpha!|2\lambda|^{-|\alpha|}|b_{\alpha}|^{2}.
\]
It is clear that the derivative $\partial_{\xi_{j}}\varphi$ and the multiplication ${\xi_{j}}\varphi$ still belong to
${\mathcal{G}}_{\lambda}$.

The Fourier transform of a function $h(z,t)$ in $L^1(\cH^{n})$ is defined by
\begin{align}
{\mathcal{F}}(h)(\lambda)=\int_{\mathbb
H^{n}}h(z,t)\pi_{z,t}^{\lambda}\,dzdt.
\end{align}
Note that $\mathcal F (h)(\lambda)$ takes its values in the space of
bounded operators on $\mathcal G_\lambda$, for every $\lambda$.

We recall two important properties of the Fourier transform:
\begin{align}\label{Fourier-Laplacian}
\tfrac{1}{2}{\mathcal{F}}(\Delta_{b}h)(\lambda)\Psi_{\alpha,\lambda}
=-(2|\alpha|+n)|\lambda|{\mathcal{F}}(h)(\lambda)\Psi_{\alpha,\lambda}
\end{align}
and
\begin{equation}\label{Fourier-t}
{\mathcal{F}}(\partial_{t}h)(\lambda)\Psi=-i\lambda{\mathcal{F}}(h)(\lambda)\Psi.
\end{equation}
The Plancherel formula is
\begin{align}\label{Plancherel}
\|h\|^2_{L^2 (\cH^{n})}=\frac{2^{n-1}}{\pi^{n+1}}
\sum_{\alpha\in\mathbb N^{n}_{0}}\int_{\R}\|
{\mathcal{F}}(h)(\lambda)\Psi_{\alpha,\lambda}\|^{2}_{{\mathcal{G}}_{\lambda}}|\lambda|^{n}\,d\lambda,
\end{align}
where Heisenberg group is endowed with a smooth left invariant measure, the Haar measure, which in the coordinate system $(x, y, t)$ is simply the Lebesgue measure $dx\,dy\,dt$. And the inversion formula,
\[
h(z,t)=
\frac{2^{n-1}}{\pi^{n+1}}\int_{\R}\text{tr}\;(\pi^{\lambda}_{z,t})^*
{\mathcal{F}}(h)(\lambda)\,|\lambda|^{n}d\lambda,
\]
where $(\pi^{\lambda}_{z,t})^*$ is the adjoint operator of $\pi^{\lambda}_{z,t}$.

\begin{Rem} In view of \eqref{Fourier-Laplacian} one may then define the (pure)
fractional powers of the sub--Laplacian by
\begin{align}\label{def-FDb}
{\mathcal{F}}((-\Delta_{b})^{\gamma}h)(\lambda)
{\Psi}_{\alpha,\lambda}
:=(2(2|\alpha|+n)|\lambda|)^{\gamma}{\mathcal{F}}(h)(\lambda)
{\Psi}_{\alpha,\lambda}.
\end{align}
However, it does not agree with the operator $P^\theta_\gamma$ we are interested in since it does not have the CR covariance property \eqref{CRcovariance}.
\end{Rem}

For simplicity, in the following we will denote
$$\hat h_\alpha(\lambda):=\mathcal{F}(h)(\lambda)\Psi_{\alpha,\lambda}\quad \mbox{and}\quad k:=k_\alpha=2|\alpha|+n,$$
and the dependence on each level $\alpha\in\N_0^n$ will be sometimes made implicit.

\subsection{Solution of the ODE}\label{subsection-Fourier}

We now perform a Fourier transform of
\eqref{scattering-equation-new}, which, in view of \eqref{Fourier-t}
and \eqref{Fourier-Laplacian}, amounts to replacing $\partial_{tt}$
by $-\lambda^2$ and $\frac12\Delta_b$ by $-|\lambda| k$. The new
equation, written in the basis $\Psi_{\alpha,\lambda}$, reduces to
$$
\left( q\partial_{qq} + (1-\gamma)\partial_q - \lambda^2 q - |\lambda| k \right) \phi = 0 \,.
$$
Here $\lambda$ and $k$ are parameters and $\phi$ is a function of
the single variable $q$. We are looking for a solution of this
equation which satisfies $\phi(q)\to 0$ as $q\to\infty$ and
$\phi(0)=1$. We make the ansatz $\phi(q) = e^{-|\lambda|q}
g(2|\lambda|q)$ and find that the equation for $\phi$ is equivalent
to the following equation for $g(x)$,
$$
x g'' + (1-\gamma - x) g' - \frac{1-\gamma+k}2 g = 0 \,.
$$
This is already Kummer's equation, but it is convenient to transform it into another equation of the same type. To do so, we set $g(x) = x^\gamma h(x)$, which leads to
$$
xh'' + (1+\gamma-x) h' - \frac{1+\gamma+k}{2} h = 0 \,.
$$
The boundary conditions become
$$\lim_{x\to 0} x^{\gamma} h(x)= 1\quad\mbox{and}\quad\lim_{x\to\infty} e^{-x/2} x^{\gamma} h(x) = 0.$$ To proceed, we recall the following facts about special functions.

\begin{Lem}[Kummer's equation]\label{kummer}
Let $a\geq 0$ and $b>1$. The equation $xw''+ (b-x)w'-aw=0$ on
$(0,\infty)$ has two linearly independent solutions $M(a,b,\cdot)$
and $V(a,b,\cdot)$. They satisfy, as $x\to\infty$,
$$
M(a,b,x) = \frac{\Gamma(b)}{\Gamma(a)} e^x x^{a-b} \left( 1+ \mathcal O(x^{-1}) \right)
$$
and
$$V(a,b,x) = x^{-a} \left( 1+ \mathcal O(x^{-1}) \right)
$$
and, as $x\to 0$,
$$
M(a,b,x) = 1+ \mathcal O(x)
\qquad\text{and}\qquad
x^{-1+b} V(a,b,x) = \frac{\Gamma(b-1)}{\Gamma(a)} + o(1) \,.
$$
The function $M(a,b,\cdot)$ is real analytic. Moreover, for all $x>0$
$$
V(a,b,x) = \frac{\pi}{\sin\pi b} \left( \frac{M(a,b,x)}{\Gamma(1+a-b)\Gamma(b)} - x^{1-b} \frac{M(1+a-b,2-b,x)}{\Gamma(a)\Gamma(2-b)} \right) \,.
$$
\end{Lem}
These facts are contained in \cite[Chp. 13]{AbSt}. The assumptions
$a\geq 0$ and $b>1$ can be significantly relaxed, but this is not
important for us. Instead of referring to the known results of this
lemma, one can directly deduce all the properties that we need in
the following from the integral representation
$$
V(a,b,x) = \frac{1}{\Gamma(a)} \int_0^\infty e^{-tx} t^{a-1} (1+t)^{b-a-1} \,dt \,.
$$
For this formula, see again \cite[Chp. 13]{AbSt}.

Let us return to our scattering problem \eqref{scattering-equation-new}. Lemma \ref{kummer} implies that
$$
h(x) = \frac{\Gamma\left( \frac{1+\gamma+k}{2}\right)}{\Gamma(\gamma)} \ V\lp \tfrac{1+\gamma+k}{2},1+\gamma,x\rp.
$$
Because of the representation of the $V$-functions in terms of the $M$-functions we learn that for $x\to 0$
\begin{align*}
h(x) & = \frac{\Gamma\left( \frac{1+\gamma+k}{2}\right)}{\Gamma(\gamma)} \ \frac{\pi}{\sin\pi (1+\gamma)} \left( \frac{M(\frac{1+\gamma+k}{2},1+\gamma,x)}{\Gamma\left( \frac{1-\gamma+k}{2}\right) \Gamma(1+\gamma)} - x^{-\gamma} \frac{M(\frac{1-\gamma+k}{2},1-\gamma,x)}{\Gamma(\frac{1+\gamma+k}{2})\Gamma(1-\gamma)} \right) \\
& = x^{-\gamma} - \frac{\Gamma(1-\gamma)}{\Gamma(1+\gamma)} \ \frac{\Gamma\left( \frac{1+\gamma+k}{2}\right)}{\Gamma\left( \frac{1-\gamma+k}{2}\right)}
+ \mathcal O(x^{1-\gamma}) \,.
\end{align*}
Undoing the substitutions we made we find that
\begin{equation}\label{asymptotics}\begin{split}
\phi(q) & = e^{-|\lambda|q} (2|\lambda| q)^\gamma h(2|\lambda| q)
= e^{-|\lambda|q} \left( 1 - \frac{\Gamma(1-\gamma)}{\Gamma(1+\gamma)} \ \frac{\Gamma\left( \frac{1+\gamma+k}{2}\right)}{\Gamma\left( \frac{1-\gamma+k}{2}\right)} (2|\lambda| q)^\gamma + \mathcal O(q) \right) \\
& = 1 - \frac{\Gamma(1-\gamma)}{\Gamma(1+\gamma)} \ \frac{\Gamma\left( \frac{1+\gamma+k}{2}\right)}{\Gamma\left( \frac{1-\gamma+k}{2}\right)} (2|\lambda| q)^\gamma + \mathcal O(q) \quad \mbox{as}\quad q\to 0.
\end{split}\end{equation}
This proves that for $0<\gamma<1$, the scattering operator which maps $F|_{q=0}\mapsto G|_{q=0}$ is diagonal with respect to the Fourier transform and its symbol is
$$
- \frac{\Gamma(1-\gamma)}{\Gamma(1+\gamma)} \ \frac{\Gamma\left( \frac{1+\gamma+k}{2}\right)}{\Gamma\left( \frac{1-\gamma+k}{2}\right)} (2|\lambda|)^\gamma \,.
$$
Thus, recalling that $T= 2\partial_t$, and using again the properties \eqref{Fourier-Laplacian}-\eqref{Fourier-t},
$$
S(s) = - |T|^\gamma \frac{\Gamma(1-\gamma)}{\Gamma(1+\gamma)} \
\frac{\Gamma\left( \frac{1+\gamma}{2} +
\frac{-\Delta_b}{2|T|}\right)}{\Gamma\left( \frac{1-\gamma}{2} +
\frac{-\Delta_b}{2|T|}\right)} \,, \quad s= \frac{m+\gamma}{2},\quad
\gamma\in (0,1) \,.$$ Taking into account \eqref{operators} and
\eqref{constant}, we have shown \eqref{symbol}.

\subsection{The extension problem}

For any $0<\gamma<1$ we may define an extension operator $\mathcal
E_\gamma$ which maps functions $f$ on the Heisenberg group $\Hei^n$
to functions $\mathcal E_\gamma f$ on the Siegel domain
$\Omega_{n+1}\simeq\Hei^n\times (0,\infty)$. For every $q>0$ we can
consider $\mathcal E_\gamma f(\cdot,q)$ as a function on the
Heisenberg group, which is defined through the Fourier multiplier
$\phi_\alpha(2|\lambda| q)$, where
$$
\phi_\alpha(x) = \frac{\Gamma\left( \frac{1+\gamma+k_\alpha}{2}\right)}{\Gamma(\gamma)} \ e^{-x/2} \ x^\gamma \ V\lp \tfrac{1+\gamma+k_\alpha}{2},1+\gamma,x\rp \,.
$$
(We do not indicate the dependence of $\phi_\alpha$ on $\gamma$ in the notation). In other words,
\begin{equation}\label{equation20}
\widehat{\mathcal E_\gamma f(\cdot,q)}_\alpha(\lambda) = \phi_\alpha(2|\lambda| q) \ \hat f_\alpha(\lambda)
\end{equation}
for every $q>0$, $\alpha\in\N^n_0$ and $\lambda\in\R$. The fact that
$\phi_\alpha(0)=1$ from the previous section implies that for $q=0$
one has, indeed,
$$
\widehat{\mathcal E_\gamma f(\cdot,0)}_\alpha(\lambda) = \hat f_\alpha(\lambda)
$$
for every $\alpha\in\N^n_0$ and $\lambda\in\R$, that is $\mathcal
E_\gamma f(\cdot,0)=f$, which justifies the name `extension'.

Moreover, the ODE facts that we established in the previous section imply that for every $\alpha\in\N^n_0$ and $\lambda\in\R$, the function $q\mapsto \widehat{\mathcal E_\gamma f(\cdot,q)}_\alpha(\lambda)$ solves the equation
$$
\left( q\partial_{qq} + (1-\gamma)\partial_q - \lambda^2 q - |\lambda| k_\alpha \right) \widehat{\mathcal E_\gamma f(\cdot,q)}_\alpha(\lambda) = 0 \,,
$$
and therefore, the function $\mathcal E_\gamma f$ on $\Omega_{n+1}$ satisfies
$$
\left( q\partial_{qq} + (1-\gamma)\partial_q + q\partial_{tt} +
\frac12\Delta_b \right) \mathcal E_\gamma f = 0 \,.
$$
This completes the proof of Theorem \ref{thm1}.


\section{Sharp Sobolev trace inequalities}\label{section-trace}

Consider the quadratic form $a_\gamma$ associated to the operator
$P_\gamma^\theta$ as defined in \eqref{5}. Using Fourier transform
it may be rewritten as
\begin{equation}\label{a_gamma}
a_\gamma[f] =\frac{2^{n-1}}{\pi^{n+1}} \sum_{\alpha\in\N_0^n} \frac{\Gamma\left( \frac{1+\gamma+k_\alpha}{2}\right)}{\Gamma\left( \frac{1-\gamma+k_\alpha}{2}\right)} \int_\R  (2|\lambda|)^\gamma \|\hat f_\alpha(\lambda)\|_{\mathcal G_\lambda}^2 \,|\lambda|^nd\lambda \,.
\end{equation}
We now consider the energy functional in the extension introduced in
\eqref{6}, where we recall that $q=\rho^2/2$,
$$
\mathcal A_\gamma[U] = 2^{1-\gamma}\int_{\Omega_{n+1}} \left(
q^{1-\gamma} |\partial_q U|^2 + q^{1-\gamma} |\partial_t U|^2 +
\frac14 q^{-\gamma} \sum_{j=1}^n \left( |X_j U|^2 + |Y_j U|^2
\right) \right) \,d\zeta
$$
for functions $U$ on $\Omega_{n+1}$. Here $d\zeta = dx_1\cdots d x_n
dy_1 \cdots dy_n dt dq$. We define the space $\dot{\mathcal
H}^{1,\gamma}(\Omega_{n+1})$ as the completion of $\mathcal
C_0^\infty(\overline{\Omega_{n+1}})$ with respect to $\mathcal
A_\gamma^\frac{1}{2}$. Here
$\overline{\Omega_{n+1}}=\Hei^n\times[0,\infty)$, including the
boundary. One can show (and it also follows essentially from our
arguments below) that this completion is a space of functions.

Similarly, we denote by $\dot S^\gamma(\Hei^n)$ the closure of $\mathcal C_0^\infty(\Hei^n)$ with respect to the quadratic form $a_\gamma$. (These are the fractional analogues of the Sobolev spaces introduced by Folland and Stein \cite{Folland-Stein:Heisenberg}). The dual space of $\dot S^\gamma(\Hei^n)$, with respect to the inner product of $L^2(\Hei^n)$, is the space $\dot S^{-\gamma}(\Hei^n)$, which is defined through $a_{-\gamma}$.

The following result contains an energy equality for the fractional norm $\dot S^\gamma(\Hei^n)$, using the extension problem from Theorem \ref{thm1}. This idea of using the extension has been successfully employed in several other settings (\cite{Frank-Lenzmann,Frank-Lenzmann-Silvestre,Banica-Gonzalez-Saez}, for instance).

\begin{Prop}\label{extequal}
Let $0<\gamma<1$ and $f\in \dot S^\gamma(\Hei^n)$. Then $\mathcal E_\gamma f\in\dot{\mathcal H}^{1,\gamma}(\Omega_{n+1})$ and
$$
\mathcal A_\gamma[\mathcal E_\gamma f] = 2^{1-\gamma}\gamma\
\frac{\Gamma(1-\gamma)}{\Gamma(1+\gamma)}\ a_\gamma[f] \,.
$$
\end{Prop}

\begin{proof}
We use the shorthand $\hat U_\alpha(\lambda,q)$ for
$\widehat{U(\cdot,q)}_\alpha(\lambda)$. In this notation,
Plancherel's identity \eqref{Plancherel} gives
\begin{equation*}\begin{split}
&\mathcal A_\gamma[U] =\\
&\,\,\, 2^{1-\gamma}\frac{2^{n-1}}{\pi^{n+1}}\sum_\alpha\! \int_\R\!
\int_0^\infty\!\! \left( q^{1-\gamma} \|\partial_q \hat
U_\alpha(\lambda,q)\|_{\mathcal G_\lambda}^2 + \left( q^{1-\gamma}
\lambda^2 + q^{-\gamma} |\lambda|k_\alpha \right) \|\hat
U_\alpha(\lambda,q)\|_{\mathcal G_\lambda}^2 \right) dq\,|\lambda|^n
d\lambda .
\end{split}\end{equation*}
We apply this to $U= \mathcal E_\gamma f$ and plug the explicit expression for $\hat
U_\alpha(\lambda,q)$ from \eqref{equation20} into the above formula. After changing variables $x=2|\lambda|q$ we arrive at
\begin{equation}\label{eq10}\mathcal A_\gamma[\mathcal E_\gamma f] =2^{1-\gamma}\frac{2^{n-1}}{\pi^{n+1}}\sum_\alpha C_{\alpha} \int_\R (2|\lambda|)^\gamma \|\hat f_\alpha(\lambda)\|_{\mathcal G_\lambda}^2 \,|\lambda|^n d\lambda\end{equation}
with the constant
$$
C_{\alpha} = \int_0^\infty \left( x^{1-\gamma} |\phi_\alpha'|^2 + \left( \frac14 x^{1-\gamma} + \frac12 x^{-\gamma} k_\alpha \right) |\phi_\alpha|^2 \right) dx \,.
$$
Its precise value will be calculated in the next few lines. According to the previous subsection,
$$
\left( x\partial_{xx} + (1-\gamma)\partial_x - \frac14 x - \frac12 k_\alpha \right)\phi_\alpha = 0 \quad\text{in}\ (0,\infty)
$$
and $\phi_\alpha(0)=1$. Moreover, $\phi_\alpha$ decays exponentially. Thus, we can multiply the equation for $\phi_\alpha$ by $x^{-\gamma}\phi_\alpha$ and integrate over the interval $(\epsilon,\infty)$ to get
$$
C_\alpha = \lim_{\epsilon\to 0} \int_\epsilon^\infty \left( x^{1-\gamma} |\phi_\alpha'|^2 + \left( \frac14 x^{1-\gamma} + \frac12 x^{-\gamma} k_\alpha \right) |\phi_\alpha|^2 \right) dx
= - \lim_{\epsilon\to 0} \epsilon^{1-\gamma} \phi_\alpha(\epsilon) \phi_\alpha'(\epsilon) \,.
$$
We know from the previous subsection that, as $\epsilon\to 0$,
$$
\phi_\alpha(\epsilon) = 1 - \frac{\Gamma(1-\gamma)}{\Gamma(1+\gamma)} \ \frac{\Gamma\left( \frac{1+\gamma+k_\alpha}{2}\right)}{\Gamma\left( \frac{1-\gamma+k_\alpha}{2}\right)} \epsilon^\gamma + \mathcal O(\epsilon)
$$
and that this expansion may be differentiated. Thus,
$$
C_\alpha = \gamma\ \frac{\Gamma(1-\gamma)}{\Gamma(1+\gamma)} \
\frac{\Gamma\left( \frac{1+\gamma+k_\alpha}{2}\right)}{\Gamma\left(
\frac{1-\gamma+k_\alpha}{2}\right)} \,.
$$
Substituting back into \eqref{eq10} and recalling \eqref{a_gamma} we conclude that
\begin{eqnarray*}
\mathcal A_\gamma[\mathcal E_\gamma f]& = &2^{1-\gamma}\gamma\
\frac{\Gamma(1-\gamma)}{\Gamma(1+\gamma)}
\frac{2^{n-1}}{\pi^{n+1}}\sum_\alpha \frac{\Gamma\left(
\frac{1+\gamma+k_\alpha}{2}\right)}{\Gamma\left(
\frac{1-\gamma+k_\alpha}{2}\right)} \int_\R (2|\lambda|)^\gamma
\|\hat f_\alpha(\lambda)\|^2_{\mathcal
G_\lambda}\,|\lambda|^nd\lambda\\
& =& 2^{1-\gamma}\gamma\ \frac{\Gamma(1-\gamma)}{\Gamma(1+\gamma)}
a_\gamma[f] \,,
\end{eqnarray*}
as claimed.\\
\end{proof}

\begin{proof}[Proof of Theorem \ref{thm2}]
Let $U\in \mathcal C_0^\infty(\overline{\Omega_{n+1}})$ and $g\in
\dot S^{-\gamma}(\Hei^n)$, the dual of $\dot S^\gamma(\Hei^n)$. Then
$h:=\big(P^\theta_{\gamma}\big)^{-1} g = P_{-\gamma}^\theta g \in
\dot S^\gamma(\Hei^n)$. As we observed before, its extension
$H=\mathcal E_\gamma h$ from \eqref{equation20} satisfies
$$
\left( q\partial_{qq} + (1-\gamma)\partial_q + q\partial_{tt} + \frac12\Delta_b \right) H = 0 \,.
$$
Using dominated convergence and \eqref{asymptotics}, see also
\eqref{7}, one can show that
$$
H(\cdot,\epsilon) \to h \ \text{in}\ \dot S^\gamma(\Hei^n)
\qquad \text{and} \qquad
 \epsilon^{1-\gamma} \frac{\partial H}{\partial q}(\cdot,\epsilon) \to -\gamma 2^{-\gamma}\ \frac{\Gamma(1-\gamma)}{\Gamma(1+\gamma)} P_\gamma^\theta h
  \ \text{in}\ \dot S^{-\gamma}(\Hei^n) \,.
$$
Thus,
\begin{equation*}\begin{split}
2^{-\gamma}\int_{\Hei^n} \overline{g(\xi)} U(\xi,0) \,d\xi & = 2^{-\gamma}\int_{\Hei^n} \overline{P^\theta_\gamma h(\xi)} U(\xi,0) \,d\xi \\
& = - \frac{1}{\gamma} \frac{\Gamma(1+\gamma)}{\Gamma(1-\gamma)} \lim_{\epsilon\to 0} \epsilon^{1-\gamma} \int_{\Hei^n} \overline{\frac{\partial H}{\partial q}(\xi,\epsilon)} U(\xi,\epsilon) \,d\xi \\
&  \begin{split}=\frac{1}{\gamma}\
\frac{\Gamma(1+\gamma)}{\Gamma(1-\gamma)} \lim_{\epsilon\to 0} &
\iint_{\Hei^n \times (\epsilon,\infty)}
\bigg( q^{1-\gamma} \overline{\partial_q H} \partial_q U + q^{1-\gamma} \overline{\partial_t H}\partial_t U \\
& \hspace{1,2cm}+\frac14 q^{-\gamma} \sum_{j=1}^n \left(
\overline{X_j H} X_j U + \overline{Y_j H} Y_j U \right) \bigg)
d\xi\,dq \,.\end{split}
\end{split}\end{equation*}
By the Schwarz inequality,
\begin{equation*}
\begin{split}
&2^{-\gamma}\left|\int_{\Hei^n} \overline{g(\xi)} U(\xi,0) \,d\xi
\right|
 \leq \frac{1}{\gamma}\ \frac{\Gamma(1+\gamma)}{\Gamma(1-\gamma)} \\
 &\times \limsup_{\epsilon\to0}\left( \iint_{\Hei^n \times (\epsilon,\infty)} \left( q^{1-\gamma} |\partial_q H|^2 + q^{1-\gamma} |\partial_t H|^2 + \frac14 q^{-\gamma} \sum_{j=1}^n \left( |X_j H|^2 + |Y_j H|^2 \right) \right) d\xi\,dq \right)^{\frac{1}{2}} \\
& \times \limsup_{\epsilon\to0}\left( \iint_{\Hei^n \times
(\epsilon,\infty)} \left( q^{1-\gamma} |\partial_q U|^2 +
q^{1-\gamma} |\partial_t U|^2 + \frac14 q^{-\gamma} \sum_{j=1}^n
\left( |X_j U|^2 + |Y_j U|^2 \right) \right) d\xi\,dq
\right)^{\frac{1}{2}} .
\end{split}\end{equation*}
Next, by Proposition \ref{extequal},
\begin{align*}
& \limsup_{\epsilon\to 0} \iint_{\Hei^n \times (\epsilon,\infty)} \left( q^{1-\gamma} |\partial_q H|^2 + q^{1-\gamma} |\partial_t H|^2 + \frac14 q^{-\gamma} \sum_{j=1}^n \left( |X_j H|^2 + |Y_j H|^2 \right) \right) d\xi\,dq \\
& \qquad= 2^{\gamma-1}\mathcal A_\gamma[H] =
\gamma\frac{\Gamma(1-\gamma)}{\Gamma(1+\gamma)}\ a_\gamma[h] =
\gamma 2^{-\gamma}\frac{\Gamma(1-\gamma)}{\Gamma(1+\gamma)}\
\big(g,P_{-\gamma}^\theta g\big) \,.
\end{align*}
Thus, we have shown that
$$
\left|\int_{\Hei^n} \overline{g(\xi)} U(\xi,0) \,d\xi \right| \leq
\left( 2^{-1+\gamma}  \frac{\Gamma(1+\gamma)}{\gamma \
\Gamma(1-\gamma)} \right)^{1/2} a_{-\gamma}[g]^{1/2} \mathcal
A_\gamma[U]^{1/2} \,.
$$
By duality, this means that $U(\cdot,0)\in \dot S^\gamma(\Hei^n)$ with
$$
a_\gamma[U(\cdot,0)] \leq
2^{-1+\gamma}\frac{\Gamma(1+\gamma)}{\gamma \ \Gamma(1-\gamma)}
\mathcal A_\gamma[U] \,.
$$
This inequality, together with the density of $\mathcal C_0^\infty(\overline{\Omega_{n+1}})$ in $\dot{\mathcal{H}}^{1,\gamma}(\Omega_{n+1})$, implies the Theorem.
\end{proof}

It was conjectured by \cite{BrFoMo} that on the CR sphere we have the following conformally invariant sharp Sobolev inequality
\begin{equation}
\label{Sobolev-inequality}\|f\|^2_{L^{q^*}(S^{2n+1})}\leq C(n,\gamma) \fint_{S^{2n+1}} f \mathcal P_{\gamma}f,\quad \text{for}\quad  q^*=\frac{2Q}{Q-2\gamma},\quad Q=2(n+1),
\end{equation}
where $\mathcal P_\gamma$ are the CR fractional powers of the Laplacian on the sphere. The fact that this inequality is valid with some constant follows
from the work of Folland and Stein \cite{Folland-Stein:Heisenberg}.
In the remarkable work \cite{Jerison-Lee:extremals} Jerison and Lee
found the optimal constant in the case $\gamma=1$. The problem of
determining the sharp constant for general $\gamma$ was solved in
\cite{Frank-Lieb}. It is also shown that in the equivalent $\mathbb
H^n$ version of \eqref{Sobolev-inequality} all optimizers are
translates, dilates or  constant multiples of the function
$$H=\lp\frac{1}{(1+|z|^2)^2+t^2}\rp^{\frac{Q-2\gamma}{4}}.$$
Putting together \eqref{Sobolev-inequality} and Theorem \ref{thm2} we complete the proof of Corollary \ref{cor-embedding}.


\section{The Yamabe problem for the conformal fractional sub--Laplacian}\label{section-Yamabe}

Let $\mathcal X$ be a $m$-dimensional complex manifold with strictly
pseudoconvex boundary $\mathcal M$. Let $g^+$ be a K\"ahler metric
on $\mathcal X$ such that there exists a globally defined
approximate solution of the Monge-Amp\`ere equation that makes $g^+$
an approximate K\"ahler-Einstein metric, with $\theta$ as contact
form on $\mathcal M$, as described in Section \ref{section-ACH}.

Fix $\gamma\in(0,1)$. The fractional Yamabe problem asks to find a contact form $\hat \theta=f^{\frac{2}{m-\gamma}}\theta$ for some $f>0$ on $\mathcal M$ such that the fractional CR curvature $Q^{\hat \theta}_\gamma$ is constant. In PDE language, we need to find a positive solution $f$ of the nonlocal problem
$$P^{\theta}_\gamma (f)=c f^{\frac{m+\gamma}{m-\gamma}}, \quad\text{on }\mathcal M.$$
From the variational point of view, we are looking for minimizers of the functional
$$I_\gamma[f]=\frac{\int_{\mathcal M} f P_\gamma^{\theta}{f} \,\theta\wedge d\theta^n}{\lp\int_{\mathcal M} |f|^{2^*}\,\theta\wedge d\theta^n\rp^{\frac{2}{2^*}}},
$$
for $2^*=\frac{2m}{m-\gamma}$.
Motivated by the Riemannian case from \cite{Gonzalez-Qing}, one may find instead minimizers of the extension functional
$$\overline I_\gamma[u]=\frac{\int_{\mathcal X} q^{m-1}|\nabla u|_{g^+}^2\,dvol_{g^+}-s(m-s)\int_{\mathcal X} q^{m-1} u^2 \,dvol_{g^+}}
{\lp\int_{\mathcal M} |u|^{2^*}\,\theta\wedge
d\theta^n\rp^{2/2^*}}.$$ In particular, in the Heisenberg group case
we may rewrite the functional as
$$\overline I_\gamma[U]=\frac{\int_{\Omega_m} \left( q^{1-\gamma} |\partial_q U|^2 + q^{1-\gamma} |\partial_t U|^2 + \frac14 q^{-\gamma} \sum_{j=1}^n \left( |X_j U|^2 + |Y_j U|^2 \right) \right) \,d\zeta}
{\lp\int_{\mathbb H^n} |U|^{2^*}\,\theta\wedge d\theta^n\rp^{2/2^*}}.$$
for $u=q^{m-s}U$.

We define the CR $\gamma$-Yamabe constant as
$$\Lambda_\gamma(\mathcal M,[\theta])=\inf{\overline I_\gamma[U]}.$$
It is easy to show that
\begin{equation}\label{compare}\Lambda_\gamma(\mathcal M,[\theta])\leq \Lambda_\gamma(\mathbb H^n),\end{equation}
where the Heisenberg group is understood with its canonical contact form.

We conjecture that the fractional CR Yamabe problem is solvable if we have a strict inequality in \eqref{compare}, and that this is so unless we are already at the model case.  We hope to return to this problem elsewhere.

\bigskip

\section{Further studies}\label{Section-further}

After all this discussion on the complex hyperbolic space, it is natural to look now at the the quaternionic hyperbolic space $\mathcal H^m_{\mathbb Q}$. It can be characterized as a Siegel domain whose boundary is precisely the quaternionic Heisenberg group that, with some abuse of notation, will be denoted by  $\mathcal Q^n$.  It will become clear that both Theorems \ref{thm1} and \ref{thm2} are consequences of the rigid underlying structure of hyperbolic space.

\subsection{The quaternionic case}

A quaternion is an object of the form
$$q=x+yi+zj+wk,\quad x,y,z,w\in\mathbb R,$$
where the three imaginary units satisfy the multiplication rules
\begin{equation*}
\begin{split}
&i^2=j^2=k^2=-1,\\
&ij=-ji=k,\quad jk=-kj=i,\quad ki=-ik=j.
\end{split}
\end{equation*}
The set of quaternions, denoted by $\mathbb Q$, is a division ring. In particular, multiplication is still associative and every nonzero element has a unique inverse. The number $x$ is called the real part of $q$ while the three dimensional vector $yi+zj+wk$ is its imaginary part. Conjugation is the same as for complex numbers, indeed, $\bar q=x-yi-zj-wk$, and the modulus of $q$ is calculated as $|q|^2=q\bar q=x^2+y^2+z^2+w^2$.

We define the quaternionic Heisenberg group by $\mathcal Q^n:=\mathbb Q^n \times \im(\mathbb Q)$, with the group law
$$(\zeta_1,v_1)\circ(\zeta_2,v_2)=(\zeta_1+\zeta_2,v_1+v_2+2\im\ll \zeta_1,\zeta_2\gg),$$
where $\ll \zeta_1,\zeta_2\gg=\overline{\zeta_2}\zeta_1$ is the standard positive definite Hermitian form on $\mathbb Q^{n}$.
If we choose coordinates
$$(\zeta_i=x_i+iy_i+jz_i+kw_i)_{i=1}^{n}\quad \text{and}\quad v=iv_1+jv_2+kv_3$$
for the group $\mathcal Q^n$, then the following 1-form is a quaternionic contact form:
\begin{equation*}
\eta=\lp\begin{array}{l}
dv_1+2\sum(x_i dy_i-y_i dx_i+z_idw_i-w_idz_i)\\
dv_2+2\sum(x_i dz_i-z_i dx_i-y_idw_i+w_idy_i)\\
dv_3+2\sum(x_i dw_i-w_i dx_i+y_idz_i-z_idy_i)
\end{array}\rp.
\end{equation*}
Note that $d\eta$ is non-degenerate and the vector fields
\begin{equation*}
\begin{split}
X_i&=\frac{\partial}{\partial x_i}+2y_i\frac{\partial}{\partial v_1}+2z_i\frac{\partial}{\partial v_2}+2w_i\frac{\partial}{\partial v_3},\\
Y_i&=\frac{\partial}{\partial y_i}-2x_i\frac{\partial}{\partial v_1}-2w_i\frac{\partial}{\partial v_2}+2z_i\frac{\partial}{\partial v_3},\\
Z_i&=\frac{\partial}{\partial z_i}+2w_i\frac{\partial}{\partial v_1}-2x_i\frac{\partial}{\partial v_2}-2y_i\frac{\partial}{\partial v_3},\\
W_i&=\frac{\partial}{\partial w_i}-2z_i\frac{\partial}{\partial v_1}+2y_i\frac{\partial}{\partial v_2}-2x_i\frac{\partial}{\partial v_3}.
\end{split}
\end{equation*}
generate the kernel of $\eta$. Then $\{X_i,Y_i,Z_i,W_i\}_{i=1}^{n}$ generate a $4n$-dimensional distribution which is a contact structure on $\mathcal Q_n$.\\

In the paper \cite{Kim-Parker:quaternionic} the quaternionic hyperbolic space $\mathcal H^m_{\mathbb Q}$ of quaternionic dimension $m$ is characterized as a Siegel domain with boundary $\mathcal Q^n$; we give here the main details of this construction. Consider $\mathbb Q^{m,1}$, the quaternionic vector space of quaternionic dimension $m+1$ (so real dimension $4m+4$) with the quaternionic Hermitian form given by
$$\langle z,w\rangle=\overline w_1 z_{m+1}+\overline w_2 z_2+\ldots+\overline w_m z_m+\overline w_{m+1}z_1,$$
where $z$ and $w$ are the column vectors in $\mathbb Q^{m,1}$ with entries $z_1,\ldots,z_{m+1}$ and $w_1,\ldots,w_{m+1}$, respectively.
Consider the subspaces $V_-,V_0,V_+$ of $\mathbb Q^{m,1}$ given by
\begin{equation*}
\begin{split}
V_-&=\{z\in\mathbb Q^{m,1}\,:\, \langle z,z\rangle<0\},\\
V_0&=\{z\in\mathbb Q^{m,1}\backslash\{0\}\,:\, \langle z,z\rangle=0\},\\
V_+&=\{z\in\mathbb Q^{m,1}\,:\, \langle z,z\rangle>0\},\\
\end{split}
\end{equation*}
Define a right projection map $P$ from the subspace of $\mathbb Q^{m,1}$ consisting of those $z$ with $z_{m+1}\neq 0$ to $\mathbb Q^m$ by
\begin{equation*}
P:\left(\begin{array}{c}z_1\\ \vdots \\z_{m+1}\end{array}\right) \to
\left(\begin{array}{c}z_1 z_{m+1}^{-1}\\ \vdots \\z_{m}z_{m+1}^{-1}\end{array}\right).
\end{equation*}
The quaternionic hyperbolic $m$ space is defined as $\mathcal H^m_{\mathbb Q}:=PV_-\subset \mathbb Q^m$. This is a paraboloid in $\mathbb Q^m$, called the Siegel domain. The metric on $\mathcal H^m_{\mathbb Q}$ is defined by
\begin{equation*}
g^+=\frac{-4}{\langle z,z\rangle^2}\det
\lp \begin{array}{cc}\langle z,z\rangle & \langle dz,z\rangle\\
\langle z,dz\rangle & \langle dz,dz\rangle
\end{array}\rp.
\end{equation*}
The boundary of this Siegel domain consists of those points in $PV_0$ (defined for points $z$ in $V_0$ with $z_{m+1}\neq 0$) together with a distinguished point at infinity, which we denote $\infty$. The finite points on the boundary of $\mathcal H^m_{\mathbb Q}$ naturally carry the structure of the generalized Heisenberg group $\mathcal Q^n$.

As in the complex case just studied, we define horospherical coordinates on quaternionic hyperbolic space. To each point $(\zeta,v,u)\in \mathcal Q^n \times \mathbb R^+$ we associate a point $\psi(\zeta,v,u)\in V_-$. Similarly, $\infty$ and each point $(\zeta,v,0)\in\mathcal Q^n \times \{0\}$ is associated to a point in $V_0$ by $\psi$. The map $\psi$ is given by
$$\psi(\zeta,v,u)=\left[
\begin{array}{c}(-|\zeta|^2-u+v)/2 \\ \zeta \\ 1\end{array}\right]\quad \text{if } z\in \overline{\mathcal H}^m_{\mathbb Q}\backslash\{\infty\},\quad \psi(\infty)=
\left[\begin{array}{c}1\\0\\ \vdots\\0\end{array}\right].$$
With these coordinates, the metric on $\mathcal H^m_{\mathbb Q}$ may be written as
$$g^+=\frac{du^2+4u\ll d\zeta,d\zeta\gg+\eta^2}{u^2},$$
which maybe put into the more standard form \eqref{metric1} by means of the change $u=\rho^2$ modulo factors of 2 depending on the normalization.
Finally, the volume form is
$$dvol_{\mathcal H^m_{\mathbb Q}}=\frac{1}{u^{2n+2}}\,du\,dvol_{\mathcal Q^n}.$$

\subsection{Asymptotically hyperbolic metrics - a general formulation}

Here we follow the notation of the book \cite{Biquard:book}. It is
well known that the rank one symmetric spaces of non-compact type
are the real, complex and quaternionic hyperbolic spaces, and the
Cayley (octonionic) hyperbolic plane.  We denote them by $\mathcal
H^m_{\mathbb K}$, where $\mathbb K=\mathbb R$, $\mathbb C$, $\mathbb
Q$ (the quaternions) or $\mathbb O$ (the octonions). As homogeneous
spaces, we may write $\mathcal H^m_{\mathbb K}=G_0/G$, where $G_0$
is a real semisimple Lie group and $G$ a maximal compact subgroup,
that is the stabilizer of a particular point $*$; more specifically,
\begin{equation*}\begin{split}
&\mathcal H^m_{\mathbb R}=SO_{1,m} / SO_m,\quad \mathcal H^m_{\mathbb C}=SU_{1,m}/U_m,\\
&\mathcal H^m_{\mathbb Q}=Sp_{1,m}/ Sp_1Sp_m,\quad \mathcal H^2_{\mathbb O}=F_4^{-20}/Spin_9.
\end{split}\end{equation*}
%
If we denote
$$d=\dimension_{\mathbb R} \mathbb K,$$
then their real dimension is
$$\dimension_{\mathbb R} \mathcal H^m_{\mathbb K}=md,$$
i.e., $m$, $2m$, $4m$ and 16 for the real, complex, quaternionic and octonionic case, respectively.

Let $r$ be the distance to $*$ and denote by $\mathbb S_r$ the
sphere of radius $r$ centered at $*$. The metric $\kappa$ on the
boundary sphere $\mathbb S$ of the hyperbolic space $\mathcal
H_{\mathbb K}^{m}$ is defined as
$$\kappa=\lim_{r\to\infty} e^{-2r}g_{\mathbb S_r}.$$
The metric is infinite except on a distribution $V$ of codimension 1 (complex case), 3 (quaternionic case) or 7 (octonionic case). In the real case it is finite and $V=T\mathbb S$. The brackets of the vector fields in $V$ generate the whole tangent bundle $T\mathbb S$, making $\kappa$ into a Carnot-Carath\'eodory metric.

Moreover, there is a contact form $\eta$ with values in $\im \mathbb K$, such that the hyperbolic metric on $\mathcal H^{m}_{\mathbb K}$ is exactly
\begin{equation}\label{metric-general}g^+=dr^2+\sinh^2 (r) \kappa+\sinh^2(2r)\eta^2.\end{equation}
In the real case, the $\eta^2$ term does not appear. To give a sense
to the formula in the other three cases, we have to choose a
supplementary subspace to the distribution $V\subset T\mathbb S$.
This is given here by the fibers of the fibration
\begin{equation*}
\begin{CD}
\mathbb S^{d-1} @>>> \mathbb S\\
@. @VVV\\
 @. \mathbb K\mathbb P^{m-1}
\end{CD}
\end{equation*}
All this depends on the choice of the  base point $*$, but the conformal class $[\kappa]$ is well defined and will be called the conformal infinity of $g^+$. Note that $g^+$ has sectional curvature pinched between -4 and -1.

Note that if in \eqref{metric-general} we make the change of variables $\rho=e^{-r}$, then the metric becomes the more familiar
\begin{equation}\label{metric1}g^+=\frac{d\rho^2+\kappa}{\rho^2}+\frac{\eta^2}{\rho^4}.\end{equation}

One could generalize these definitions to give a notion of asymptotically $\mathbb K$-hyperbolic manifolds, whose metric behaves asymptotically at infinity as \eqref{metric-general}, but we will not pursue this end further.

\subsection{Scattering theory on $\mathcal H^m_{\mathbb Q}$}

Scattering theory on asymptotically $\mathbb K$-hyperbolic manifolds was developed in \cite{Biquard-Mazzeo,Mazzeo-Vasy:symmtric-spaces} (see also \cite{Carron-Pedon,Pedon:quaternionic} for the generalization to differential forms). Here we would like to show that in the case of hyperbolic space $\mathcal H^m_{\mathbb Q}$, the calculation of the conformal fractional Laplacian $P_\kappa^\eta$ from Theorem \ref{thm1} and the energy identity from Theorem \ref{thm2} are analogous.

Denote $m_0=2+4m$.  We calculate
$$\Delta_{+}=\rho^2\partial_{\rho\rho}-(1+4m)\rho\partial\rho+\rho^2\Delta_\kappa+\rho^4\Delta_\eta.$$
It is well known that the bottom of the spectrum for $-\Delta_+$ is $(m_0/2)^2=(1+2m)^2$. The scattering equation
\begin{equation}\label{scattering1}-\Delta_+ u-s(m_0-s)u=0,\quad\text{in }\mathcal H^m_{\mathbb Q}\end{equation}
has two indicial roots $s$ and $m_0-s$, and one seeks a solution
$$u=F\rho^{m_0-s}+G\rho^s, \quad F|_{\rho=0}=f.$$

Change variables $U=\rho^{m_0-s}u$, and set $s=\frac{m_0}{2}+\gamma$, $a=1-2\gamma$. Then equation \eqref{scattering1} becomes
\begin{equation*}
\left\{\begin{split}
\partial_{\rho\rho}U+ \frac{a}{\rho}\partial_{\rho}U+\Delta_\kappa U+\rho^2\Delta_\eta U&=0,\\
U&=f,
\end{split}\right.\end{equation*}
which is precisely \eqref{equation-CafSil}. From here we can easily
produce (quaternionic--conformal) fractional powers for the
sub--Laplacian $\Delta_\kappa$. We leave the details to the
interested reader.

Some final references: on harmonic analysis on semisimple Lie groups and symmetric spaces
see \cite{Cowling-Koranyi,Cowling:applications,Johnson-Wallach} and the books by Helgason
\cite{Helgason:libro,Helgason:libro2,Helgason:libro3}. On the
quaternionic Yamabe problem see \cite{Ivanov-Vassilev:book}. However, the
problem of finding extremals for fractional Sobolev embeddings in this setting is still an
open question.


\bigskip
\noindent {\bf Acknowledgements:} R.F. acknowledges financial
support from the NSF grants PHY-1068285 and PHY-1347399. M.G. was supported by Spain
Government grant MTM2011-27739-C04-01 and GenCat 2009SGR345. D.M.
was supported by GNAMPA project with title ``Equazioni differenziali con
invarianze in analisi globale'', by GNAMPA section ``Equazioni
differenziali e sistemi dinamici'' and by MIUR project ``Metodi
variazionali e topologici nello studio di fenomeni nonlineari''.
J.T. was supported by Chile Government grants Fondecyt \#1120105,
the Spain Government grant MTM2011-27739-C04-01 and Programa Basal,
CMM. U. de Chile.


\begin{thebibliography}{99}

\bibitem{AbSt} M. Abramowitz, I. A. Stegun, \textit{Handbook of mathematical functions with formulas, graphs, and mathematical tables}. Reprint of the 1972 edition. Dover Publications, New York, 1992.

\bibitem{Applebaum-Cohen}
D.~Applebaum and S.~Cohen.
\newblock {\em L\'evy processes, pseudo-differential operators and {D}irichlet forms
  in the {H}eisenberg group}.
\newblock Ann. Fac. Sci. Toulouse Math. (6), 13(2):149--177, 2004.

\bibitem{Bahouri-Gallagher:Heat}
H.~Bahouri and I.~Gallagher.
\newblock The heat kernel and frequency localized functions on the {H}eisenberg
  group.
\newblock In {\em Advances in phase space analysis of partial differential
  equations}, volume~78 of {\em Progr. Nonlinear Differential Equations Appl.},
  pages 17--35. Birkh\"auser Boston Inc., Boston, MA, 2009.


\bibitem{Banica-Gonzalez-Saez}
V. Banica, M.d.M. Gonz\'alez and M. S\'aez.
\newblock {\em Some constructions for the fractional Laplacian on noncompact manifolds}.
\newblock Preprint.

\bibitem{Beals-Fefferman-Grossman}
M. Beals, C. Fefferman, and R. Grossman.
\newblock \emph{Strictly pseudoconvex domains in} {${\bf C}^{n}$}.
\newblock  Bull. Amer. Math. Soc. (N.S.), 8(2):125--322, 1983.

\bibitem{Biquard:book}
O.~Biquard.
\newblock {\em M\'etriques d'{E}instein asymptotiquement sym\'etriques}.
\newblock Ast\'erisque, (265), 2000.

\bibitem{Biquard-Mazzeo}
O.~Biquard and R.~Mazzeo.
\newblock {\em A nonlinear {P}oisson transform for {E}instein metrics on product
  spaces}.
\newblock J. Eur. Math. Soc. (JEMS), 13(5):1423--1475, 2011.

\bibitem{BrFoMo}
T.~P. Branson, L.~Fontana, and C.~Morpurgo.
\newblock Moser-{T}rudinger and {B}eckner-{O}nofri's inequalities on the {CR}
  sphere.
\newblock {\em Ann. of Math. (2)}, 177(1):1--52, 2013.


\bibitem{Branson-Olafsson-Orsted}
T.~Branson, G.~{\'O}lafsson, and B.~{\O}rsted.
\newblock {\em Spectrum generating operators and intertwining operators for
  representations induced from a maximal parabolic subgroup}.
\newblock J. Funct. Anal., 135(1):163--205, 1996.

 \bibitem{CaSi} L. Caffarelli, L. Silvestre, \emph{An extension problem related to the fractional
 Laplacian}, Comm. in Part. Diff. Equa. 32 (2007), 1245-1260.

\bibitem{Carron-Pedon}
G.~Carron and E.~Pedon.
\newblock {\em On the differential form spectrum of hyperbolic manifolds}.
\newblock Ann. Sc. Norm. Super. Pisa Cl. Sci. (5), 3(4):705--747, 2004.


\bibitem{Case-Yang}
J. Case, P. Yang. \emph{A Paneitz-type operator for CR pluriharmonic functions}. Bull. Inst. Math. Acad. Sin. (N.S.), 8(3):285--322, 2013.

\bibitem{Chang-Chang-Tie:Paneitz}
D-C. Chang, S-C. Chang, and J. Tie.
\newblock \emph{Laguerre calculus and {P}aneitz operator on the {H}eisenberg group}.
\newblock Sci. China Ser. A, 52(12):2549--2569, 2009.


\bibitem{Chang-Gon}
  A. Chang and M.d.M. Gonz\'alez.
  \emph{Fractional Laplacian in conformal geometry},
  Advances in Mathematics 226, 1410--1432, 2011.

\bibitem{Cheng-Yau:Kahler}
S.~Y. Cheng and S.~T. Yau.
\newblock {\em On the existence of a complete {K}\"ahler metric on noncompact
  complex manifolds and the regularity of {F}efferman's equation.}
\newblock Comm. Pure Appl. Math., 33(4):507--544, 1980.


\bibitem{Cowling:applications}
M.~Cowling.
\newblock Applications of representation theory to harmonic analysis of {L}ie
  groups (and vice versa).
\newblock In {\em Representation theory and complex analysis}, volume 1931 of
  {\em Lecture Notes in Math.}, pages 1--50. Springer, Berlin, 2008.

\bibitem{Cowling-Koranyi}
M.~Cowling and A.~Kor{\'a}nyi.
\newblock Harmonic analysis on {H}eisenberg type groups from a geometric
  viewpoint.
\newblock In {\em Lie group representations, {III} ({C}ollege {P}ark, {M}d.,
  1982/1983)}, volume 1077 of {\em Lecture Notes in Math.}, pages 60--100.
  Springer, Berlin, 1984.

\bibitem{f}  C. Fefferman, \emph{Monge-Ampere equations, the Bergman kernel, and geometry of pseudoconvex domains},
 Ann. of Math. (2) 103, no. 2, 395--416, 1976.


 \bibitem{delRio03}
 H. del Rio, S. Simanca,
  \emph{The Yamabe problem for almost Hermitian manifolds}, J. Geom. Anal. 13 (2003) 185--203.

\bibitem{Dragomir-Tomassini}
S. Dragomir and G. Tomassini.
\newblock {\em Differential geometry and analysis on {CR} manifolds}, volume
  246 of {\em Progress in Mathematics}.
\newblock Birkh\"auser Boston Inc., Boston, MA, 2006.

\bibitem{Epstein91} C. Epstein, R. Melrose, G. Mendoza,   \emph{Resolvent of the Laplacian on pseudoconvex
domains}, Acta Math. 167:1-106,  1991.

\bibitem{Fefferman:Monge-Ampere}
C. Fefferman.
\newblock \emph{Monge-{A}mp\`ere equations, the {B}ergman kernel, and geometry of
  pseudoconvex domain}s.
\newblock Ann. of Math. (2), 103(2):395--416, 1976.
\newblock \emph{Correction}: Ann. of Math. (2) 104, no. 2:393�394, 1976.


\bibitem{Fefferman-Hirachi}
C. Fefferman and K. Hirachi.
\newblock \emph{Ambient metric construction of {$Q$}-curvature in conformal and {CR}
  geometries.}
\newblock Math. Res. Lett., 10(5-6):819--831, 2003.

\bibitem{Folland-Stein:Heisenberg}
G. Folland and E. Stein.
\newblock {\em Estimates for the {$\bar \partial _{b}$} complex and analysis on the
  {H}eisenberg group.}
\newblock Comm. Pure Appl. Math., 27:429--522, 1974.

\bibitem{FraFe} B. Franchi, F. Ferrari, \emph{Harnack inequality for fractional laplacians in Carnot groups}, preprint.

\bibitem{Frank-Lieb} R. Frank and E. Lieb, \emph{Sharp constants in several inequalities on the
    Heisenberg group}, Ann. of Math. 176(1):349 -- 381, 2012.

\bibitem{Frank-Lenzmann}
R. Frank and  E. Lenzmann, \emph{Uniqueness of non-linear ground states for fractional Laplacians in $\mathbb R$}. Acta Math. 210, no.
2: 261--318, 2013

\bibitem{Frank-Lenzmann-Silvestre}
R. Frank, E. Lenzmann and L. Silvestre,  \emph{Uniqueness of radial solutions for
the fractional Laplacian}. Preprint, arxiv:1302.2652.

\bibitem{GarLan}
N. Garofalo and E. Lanconelli, \emph{Frequency functions on the Heisenberg
group, the uncertainty principle and unique continuation}, Ann. Inst.
Fourier, Grenoble 40(2): 313--356, 1990.

\bibitem{Geller77} D. Geller, \emph{Fourier analysis on the Heisenberg group}, Proc. Nat. Acad. Sci. U.S.A., 74:1328-1331, 1977.

\bibitem{Goldman:book}
W. Goldman.
\newblock {\em Complex hyperbolic geometry}.
\newblock Oxford Mathematical Monographs. The Clarendon Press Oxford University
  Press, New York, 1999.
\newblock Oxford Science Publications.

\bibitem{Gonzalez-Qing}
M.d.M. Gonz\'alez and J. Qing, \emph{Fractional conformal Laplacians and fractional Yamabe problems}. To appear in Analysis and PDE.

\bibitem{Gover-Graham:CR-powers}
A. Gover and C.~R. Graham.
\newblock {\em C{R} invariant powers of the sub-{L}aplacian}.
\newblock J. Reine Angew. Math., 583:1--27, 2005.

\bibitem{Graham:thesisI}
C.~R. Graham.
\newblock \emph{The {D}irichlet problem for the {B}ergman {L}aplacian. {I}.}
\newblock Comm. Partial Differential Equations, 8(5):433--476, 1983.

\bibitem{Graham:thesisII}
C.~R. Graham.
\newblock \emph{The {D}irichlet problem for the {B}ergman {L}aplacian. {II}}.
\newblock Comm. Partial Differential Equations, 8(6):563--641, 1983.

\bibitem{Graham:compatibility}
C.R. Graham.
\newblock \emph{Compatibility operators for degenerate elliptic equations on the ball and {H}eisenberg group}.
\newblock Math. Z., 187(3):289--304, 1984.

\bibitem{Graham-Jenne-Mason-Sparling}
C.~R. Graham, R.~Jenne, L.~J. Mason, and G.~A.~J. Sparling.
\newblock \emph{Conformally invariant powers of the {L}aplacian. {I}. {E}xistence.}
\newblock J. London Math. Soc. (2), 46(3):557--565, 1992.

\bibitem{GrahamLee} C.R. Graham, J. Lee, \emph{Smooth solutions of degenerate Laplacians on strictly pseudoconvex domains},
 Duke Math. J. 57  (1988) 697--720.

\bibitem{GZ03} C.R. Graham, M. Zworski, \emph{Scattering matrix in conformal geometry}, Invent. Math. 152 (1) (2003) 89�-118.

\bibitem{Guillarmou08} C. Guillarmou, A. Sa Barreto,
 \emph{Scattering and inverse scattering on ACH manifolds},
  J. Reine Angew. Math. 622 (2008) 1--55.

\bibitem{Helgason:libro3}
S.~Helgason.
\newblock {\em Groups and geometric analysis}, volume~83 of {\em Mathematical
  Surveys and Monographs}.
\newblock American Mathematical Society, Providence, RI, 2000.
\newblock Integral geometry, invariant differential operators, and spherical
  functions, Corrected reprint of the 1984 original.

\bibitem{Helgason:libro2}
S.~Helgason.
\newblock {\em Differential geometry, {L}ie groups, and symmetric spaces},
  volume~34 of {\em Graduate Studies in Mathematics}.
\newblock American Mathematical Society, Providence, RI, 2001.
\newblock Corrected reprint of the 1978 original.

\bibitem{Helgason:libro}
S.~Helgason.
\newblock {\em Geometric analysis on symmetric spaces}, volume~39 of {\em
  Mathematical Surveys and Monographs}.
\newblock American Mathematical Society, Providence, RI, second edition, 2008.

\bibitem{Hirachi:Q-prime}
K. Hirachi.
\newblock \emph{Q-prime curvature on {CR} manifolds.}
\newblock Preprint, 2013.

\bibitem{HislopPeterTang}
P. Hislop,  P. Perry and  S. Tang,  \emph{CR-invariants and the scattering operator for complex manifolds with boundary},
 Anal. PDE 1 (2008)  197--227.

\bibitem{Hormander}
L. H{\"o}rmander.
\newblock {\em Hypoelliptic second order differential equations}.
\newblock Acta Math., 119:147--171, 1967.

\bibitem{Ivanov-Vassilev:book}
S. Ivanov and D. Vassilev.
\newblock {\em Extremals for the {S}obolev inequality and the quaternionic
  contact {Y}amabe problem}.
\newblock World Scientific Publishing Co. Pte. Ltd., Hackensack, NJ, 2011.

\bibitem{Jerison-Lee:Yamabe}
D. Jerison and J. Lee.
\newblock {\em The {Y}amabe problem on {CR} manifolds}.
\newblock J. Differential Geom., 25(2):167--197, 1987.

\bibitem{Jerison-Lee:extremals}
D. Jerison and J. Lee.
\newblock {\em Extremals for the {S}obolev inequality on the {H}eisenberg group and
  the {CR} {Y}amabe problem}.
\newblock J. Amer. Math. Soc., 1(1):1--13, 1988.

\bibitem{Jerison-Lee:coordinates}
D. Jerison and J. Lee.
\newblock {\em Intrinsic {CR} normal coordinates and the {CR} {Y}amabe problem}.
\newblock J. Differential Geom., 29(2):303--343, 1989.

\bibitem{Johnson-Wallach}
K. Johnson and N. Wallach.
\newblock {\em Composition series and intertwining operators for the spherical
  principal series. {I}.}
\newblock Trans. Amer. Math. Soc., 229:137--173, 1977.

\bibitem{Kim-Parker:quaternionic}
I. Kim and J. Parker.
\newblock {\em Geometry of quaternionic hyperbolic manifolds.}
\newblock Math. Proc. Cambridge Philos. Soc., 135(2):291--320, 2003.

\bibitem{La} N. Landkof,
\emph{Foundations of Modern Potential Theory}, Springer-Verlag, 1972.

\bibitem{Lee:pseudoeinstein}
J. Lee.
\newblock \emph{Pseudo-{E}instein structures on {CR} manifolds}.
\newblock Amer. J. Math., 110(1):157--178, 1988.

\bibitem{Lee-Melrose}
J. Lee and R. Melrose.
\newblock \emph{Boundary behaviour of the complex {M}onge-{A}mp\`ere equation}.
\newblock Acta Math., 148:159--192, 1982.

\bibitem{Li-Monticelli}
Y.~Y. Li and D. Monticelli.
\newblock \emph{On fully nonlinear {CR} invariant equations on the {H}eisenberg
  group}.
\newblock J. Differential Equations, 252(2):1309--1349, 2012.

\bibitem{Lopez-Sire} L. L\'opez and Y. Sire. \emph{Rigidity results for non local  phase transitions in the Heisenberg group $\mathbb H$}. Preprint.

\bibitem{MalcUguz}
A. Malchiodi and F. Uguzzoni, \emph{A perturbation result for the Webster
scalar curvature problem on the CR sphere}, J. Math. Pures Appl. 81
(2002), 983�-997.

\bibitem{Mazzeo-Vasy:symmtric-spaces}
R.~Mazzeo and A.~Vasy.
\newblock {\em Analytic continuation of the resolvent of the {L}aplacian on
  symmetric spaces of noncompact type.}
\newblock J. Funct. Anal., 228(2):311--368, 2005.

\bibitem{Pedon:quaternionic}
E.~Pedon.
\newblock {\em The differential form spectrum of quaternionic hyperbolic spaces}.
\newblock Bull. Sci. Math., 129(3):227--265, 2005.

\bibitem{Seshadri:volume}
N. Seshadri.
\newblock \emph{Volume renormalization for complete {E}instein-{K}\"ahler metrics}.
\newblock Differential Geom. Appl., 25(4):356--379, 2007.

\bibitem{Taylor:book}
M. Taylor.
\newblock \emph{Noncommutative harmonic analysis},
\newblock Mathematical Surveys and Monographs 22, American Mathematical Society, Providence, RI,1986.

\bibitem{Thangavelu:book}
S. Thangavelu.
\newblock {\em Harmonic analysis on the {H}eisenberg group}, volume 159 of {\em
  Progress in Mathematics}.
\newblock Birkh\"auser Boston Inc., Boston, MA, 1998.

\end{thebibliography}
\end{document}